\pdfoutput=1
\documentclass[11pt, a4paper, oneside, reqno]{amsart}
\usepackage[english]{babel}
\usepackage[T1]{fontenc}
\swapnumbers
\usepackage{booktabs}

\newcommand{\bfA}{{\bf A}}
\newcommand{\bfB}{{\bf B}}
\usepackage[usenames,dvipsnames]{xcolor}
\usepackage[colorlinks=true,linkcolor=NavyBlue,urlcolor=RoyalBlue,citecolor=PineGreen,
	hypertexnames=false]{hyperref}
\usepackage{amsfonts}
\usepackage{amsmath}
\usepackage{amssymb}
\usepackage{amsthm}
\usepackage{latexsym}
\usepackage{stmaryrd}
\usepackage{quiver}
\usepackage{mathtools}
\usepackage{bm}

\usepackage{mathbbol}
\usepackage{shuffle}
\usepackage{rotating}
\usepackage{tikz}
\usepackage{tikz-cd} \usepackage{subfigure}
\usepackage{verbatim}
\usepackage[margin=0.7in]{geometry}
\usepackage[font=small]{caption}
\usepackage{adjustbox}
\usepackage{enumerate}
\usepackage{cleveref}
\usepackage{svg}
\usepackage{tcolorbox}
\usepackage[tracking=smallcaps,expansion=alltext,protrusion=true]{microtype}\SetTracking[spacing={25*,166,}]{encoding=*,shape=sc}{50}

\usepackage[linecolor=white,backgroundcolor=white,bordercolor=white,textsize=tiny]{todonotes}

\usepackage[linecolor=white,backgroundcolor=white,bordercolor=white,textsize=tiny]{todonotes}

\makeatletter
\renewcommand{\tocsection}[3]{%
	\indentlabel{\@ifnotempty{#2}{\bfseries\ignorespaces#1 #2.\,\,}}\bfseries#3}
\renewcommand{\tocsubsection}[3]{%
	\indentlabel{\@ifnotempty{#2}{\ignorespaces#1 #2\quad}}#3}
\renewcommand{\tocsubsubsection}[3]{%
	\quad\quad\quad\indentlabel{\@ifnotempty{#2}{\ignorespaces#1 #2\quad}}#3}

\newcommand\@dotsep{4.5}
\def\@tocline#1#2#3#4#5#6#7{\relax
	\ifnum #1>\c@tocdepth 
	\else
		\par \addpenalty\@secpenalty\addvspace{#2}%
		\begingroup \hyphenpenalty\@M
		\@ifempty{#4}{%
			\@tempdima\csname r@tocindent\number#1\endcsname\relax
		}{%
			\@tempdima#4\relax
		}%
		\parindent\z@ \leftskip#3\relax \advance\leftskip\@tempdima\relax
		\rightskip\@pnumwidth plus1em \parfillskip-\@pnumwidth
		#5\leavevmode\hskip-\@tempdima{#6}\nobreak
		\leaders\hbox{$\m@th\mkern \@dotsep mu\hbox{.}\mkern \@dotsep mu$}\hfill
		\nobreak
		\hbox to\@pnumwidth{\@tocpagenum{\ifnum#1=1\bfseries\fi#7}}\par
		\nobreak
		\endgroup
	\fi}
\AtBeginDocument{%
	\expandafter\renewcommand\csname r@tocindent0\endcsname{0pt}
}

\def\l@subsection{\@tocline{2}{0pt}{2.5pc}{5pc}{}}
\makeatother

\DeclareFontFamily{U}{MnSymbolC}{}
\DeclareFontShape{U}{MnSymbolC}{m}{n}{
	<-5.5> MnSymbolC5
	<5.5-6.5> MnSymbolC6
	<6.5-7.5> MnSymbolC7
	<7.5-8.5> MnSymbolC8
	<8.5-9.5> MnSymbolC9
	<9.5-11.5> MnSymbolC10
	<11.5-> MnSymbolCb12
}{}

\let\todon\todo
\renewcommand{\todo}[1]{\todon[color=green!40]{\color{NavyBlue}{#1}}}

\def\1{\mathbf{1}}

\newtheoremstyle{bfnote}%
{}{}%
{\itshape}{}%
{\sf}{.}%
{ }%
{\thmname{#1}\thmnumber{ #2}\thmnote{ (#3)}}

\theoremstyle{theorem}
\newtheorem{theorem}{Theorem}[subsection]
\newtheorem{example}[theorem]{Example}
\newtheorem{proposition}[theorem]{Proposition}
\newtheorem*{proposition*}{Proposition}
\newtheorem{corollaire}[theorem]{Corollary}
\newtheorem*{theorem*}{Theorem}
\newtheorem{lemma}[theorem]{Lemma}
\newtheorem{remark}[theorem]{Remark}
\newtheorem{definition}[theorem]{Definition}

\newcommand{\Mm}{\mathcal{M}_{N}(\mathbb{C})}

\newcommand{\trn}{\mathrm{tr}_{N}}

\newcommand{\Tr}{\mathrm{Tr}}
\newcommand{\Trn}{\mathrm{Tr}_{N}}
\newcommand{\Trnm}{\mathrm{Tr}_{N-1}}

\author[G. C\'ebron]{Guillaume C\'ebron}\thanks{G.C. is supported by the Project MESA (ANR-18-CE40-006) and by the Project STARS (ANR-20-CE40-0008) of the French National Research Agency (ANR)}
\address{Institut de Mathématiques de Toulouse; UMR5219; Université de Toulouse; CNRS; UPS, F-31062 Toulouse, France}
\email{guillaume.cebron@math.univ-toulouse.fr}
\author[N. Gilliers]{Nicolas Gilliers}\thanks{N.G. is supported by the Project STARS (ANR-20-CE40-0008) of the French National Research Agency (ANR)}
\address{Institut de Mathématiques de Toulouse; UMR5219; Université de Toulouse; CNRS; UPS, F-31062 Toulouse, France}
\email{nicolas.gilliers@math.univ-toulouse.fr}
\thanks{We thank Octavio Arizmendi for useful discussions about cyclic independences}
\title{Asymptotic cyclic-conditional freeness of random matrices}
\usepackage[sc]{mathpazo}
\begin{document}
\maketitle
\begin{abstract}
Voiculescu's freeness emerges in computing the asymptotic of spectra of polynomials on $N\times N$ random matrices with eigenspaces in generic positions: they are randomly rotated with a uniform unitary random matrix~$U_N$. In this article we elaborate on the previous point by proposing a random matrix model, which we name the \emph{Vortex model}, where $U_N$ has the law of a uniform unitary random matrix conditioned to leave invariant one deterministic vector~$v_N$. In the limit $N \to +\infty$, we show that $N\times N$ matrices randomly rotated by the matrix $U_N$ are \emph{asymptotically conditionally free} with respect to the normalized trace and the state vector $v_N$. To describe second order asymptotics, we define \emph{cyclic-conditional freeness}, a new notion of independence unifying \emph{infinitesimal freeness}, \emph{cyclic-monotone independence} and \emph{cyclic-Boolean independence}. The infinitesimal distribution in the Vortex model can be computed thanks to this new independence. Finally, we elaborate on the Vortex model in order to build random matrix models for \emph{ordered freeness} and for \emph{indented independence}.
\end{abstract}

\tableofcontents

\section{Introduction}

The main contributions of this article are two folds. Our first concern deals with a matricial model displaying asymptotic conditional freeness (both scalar and operator-valued), presented below under the name of \emph{Vortex Model} for the normalized trace. 
Analyzing more in details this model, in particular the infinitesimal distributions, leads us to introducing a new notion of non-commutative independence generalizing cyclic-Boolean, cyclic-monotone and infinitesimal freeness. This makes our second point.


\subsection{Background}

Let $\mathcal{M}_N(\mathbb{C})$ be the set of all $N\times N$ matrices with complex entries, and consider a sequence of  deterministic matrices $A_N \in \mathcal{M}_N(\mathbb{C})$ and $B_N \in \mathcal{M}_{N}(\mathbb{C})$ $(N\geq 1)$ \emph{bounded in operator-norm} uniformly in $N$. Given a sequence of uniform unitary random matrices $U_N$ (i.e. distributed according to Haar measure on the compact group $U(N)$ of all \emph{complex} unitary matrices), Voiculescu’s asymptotic freeness \cite{voiculescu1991limit} for random matrices states that, almost surely, the matrices $A_N$ and $U_NB_NU_N^\star$ are asymptotically free with respect to the normalized trace $\trn:M\mapsto \frac{1}{N}\Trn(M)$ as $N \to \infty$. Recently, Dahlqvist, Gabriel and the first author \cite{cebron2022freeness} obtained an extension to this result by considering a vector state $\varphi^{v_N}:M\mapsto \langle M v_N,v_N\rangle$ for a sequence of deterministic vectors $v_N\in \mathbb{C}^N$.
Interestingly, they observed that in order to describe the asymptotic behavior of $A_N$ and $U_NB_NU_N^\star$ with respect to $\varphi^{v_N}$, \emph{conditional freeness} (or for short \emph{c-freeness}) defined by Bo{\.z}ejko and Speicher \cite{bozejko1991independent,bozejko1996convolution} is pertinent: almost surely, the matrices $A_N$ and $U_NB_NU_N^\star$ are asymptotically c-free with respect to $(\trn,\varphi^{v_N})$ as $N \to \infty$. The distributions of $A_N$ with respect to $\trn$ and to $\varphi^{v_N}$ can of course be very different, in contrast with the fact that the distributions of $U_NB_NU_N^\star$ with respect to $\trn$ and to $\varphi^{v_N}$ can not be distinguished as $N \to \infty$. This is the main limitations of the above mentioned result. In fact, this "incomplete" random matrix model of c-freeness yields a random matrix model displaying asymptotic \emph{monotone independence} of Muraki~\cite{muraki2001monotonic} whenever one sequence of matrices has eigenvalues accumulating at $0$, but certainly not \emph{Boolean independence} of Bo{\.z}ejko~\cite{bozejko1986positive}. It brings about the question of finding a random matrix model for c-freeness overcoming these limitations.

On the other hand, Dahlqvist, Male and the first author \cite[Example 9.3]{cebron2016traffic} has shown that Boolean independence  with respect to $\varphi^{v_N}$ can emerge asymptotically for randomly rotated matrices $U_NB_NU_N^\star$ under the condition that the unitary matrix $U_N$ involved leaves invariant $v_N$ (in \cite[Example 9.3]{cebron2016traffic}, the particular case under study is the case where $U_N$ is a permutation matrix and $v_N$ is a scaled all-ones vector). This observation provides a good hint at a random matrix model displaying asymptotic c-freeness: one of the state (the $\psi$-state under the notations in use in \cite{bozejko1991independent} should be the asymptotic joint distribution of $A_N$ and $U_NB_NU_N^\star$ with respect to vector-state $\varphi^{v_N}$ where $U_N$ is a uniform unitary random matrix conditioned to leave invariant $v_N$. 

The concept of \emph{cyclic-monotone independence} was introduced by Collins, Hasebe and Sakuma in \cite{collins2018free} (see also \cite{arizmendi2022cyclic} for a modified version restoring associativity). It is apparently similar to monotone independence, but involves \emph{two} tracial linear functionals: one of the two is a state and no assumptions are made on the second (recall that for monotone independence, one deals with one state). Let $({\bf A}_N)_{N\geq 1}$ and $({\bf B}_N)_{N\geq 1}$ be two sequences of families of $N\times N$ deterministic matrices, bounded in operator-norm. If we assume that the distribution of ${\bf A}_N$ with respect to the normalized trace $\trn$ converges and that the distribution of  ${\bf B}_N$ with respect to the un-normalized trace $\Trn$ converges as $N$ tends to infinity, then the un-normalized trace of any monomial in ${\bf A}_N$ and $U_N{\bf B}_NU_N^{\star}$ converges provided that it contains at least one matrix in the ensemble ${\bf A}_N$. Cyclic-monotone independence is a set of algebraic rules that were designed to abstract computation of asymptotic of the mixed moments in ${\bf A}_N$ and $U_N{\bf B}_NU_N^{\star}$.

In \cite{arizmendi2022cyclic}, Arizmendi, Hasebe and Lehner introduced an operatorial model to cyclic-monotone independence, which they leverage to introduce the new concept of \emph{cyclic-boolean independence}. They proceeded with defining the appropriate linearization transforms (cumulants) together with a central limit theorem.

We move on to describing our model, the \emph{Vortex model}.

\subsection{The Vortex model}Let $v_N \in \mathbb{C}^N$ (with $N>1$) be a deterministic sequence of unitary vectors.
The matricial model at stakes in this article deals with the asymptotics in high dimension of $A_N$ and $U_N B_N U_N^\star$ (with $N>1$), where $A_N$ and $B_N$ are matrices drawn as above from sequences of ensembles of $N\times N$ deterministic matrices ${\bf A}_N$ and ${\bf B}_N$ bounded in operator-norm but $U_N$ is now a sequence of \emph{uniform unitary random matrices conditioned to leave invariant} $v_N$. The group of unitary matrices of size $N\times N$ leaving invariant $v_N$ is isomorphic (via restriction/corestriction to the orthogonal complement of $\langle v_N\rangle$) to the group of unitary matrices of size $(N-1)\times (N-1)$ and is therefore equipped with Haar measure. Our sequence of unitary matrices $U_N$ is distributed according to this Haar measure.
This model is called the \emph{Vortex model} as the whole space is rotated around one axis given by $v_N$. 


\subsection{Main results}
We prove the following results about the Vortex model:
\begin{enumerate}
\item The ensemble of matrices ${\bf A}_N$ and the randomly rotated ensemble of matrices $ U_N{\bf B}_NU^{\star}_N$ are \emph{asymptotically c-free} with respect to $(\trn,\varphi^{v_N})$ as $N \to \infty$ (see Theorem~\ref{th:cfreeness_almost_sure}). More precisely, for any polynomial $P$ in non-commuting variables, we prove that, as $N$ tends to infinity,
$$
\frac{1}{N}\mathrm{Tr}[P({\bf A}_N,U_N{\bf B}_NU_N^{\star})] = \psi_{{\bf A}_N} * \psi_{{\bf B}_N}(P) + o(1)
$$
almost surely, where $*$ denotes the free product of unital linear functionals (see Section~\ref{Def:freeness}), and
$$
\langle P({\bf A}_N,U_N{\bf B}_NU^{\star}_N) v_N, v_N \rangle = \varphi_{{\bf A}_N}^{v_N} \prescript{}{\psi_{{\bf A}_N}\!\!}{*}^{}_{\psi_{{\bf B}_N}} \varphi_{{\bf B}_N}^{v_N}(P) + o(1)
$$
almost surely, where $\prescript{}{\psi_{{\bf A}_N}\!\!}{*}^{}_{\psi_{{\bf B}_N}}$ denotes the c-free product of unital linear functionals (see Section~\ref{Def:freeness}).
\item We introduce a new independence we name \emph{cyclic-conditional freeness} (or \emph{cyclic c-freeness} for short) between triples of linear functionals. We recover as special cases \emph{infinitesimal freeness} of Belinschi and Shlyakhtenko \cite{belinschi2012free},   \emph{cyclic-monotone independence} of Collins, Hasebe and Sakuma \cite{collins2018free} and \emph{cyclic-Boolean independence} of Arizmendi, Hasebe and Lehner \cite{arizmendi2022cyclic} (see Section~\ref{Sec:cfreeness} and Figure~\ref{Figureone}). Cyclic-conditional independence yields a new associative product on unital algebras equipped with two unital linear forms $\psi,\varphi$ and a \emph{tracial} linear form $\omega$,
\begin{equation*}
    (\mathcal{A}_1,\psi_1,\varphi_1,\omega_1)\circledast (\mathcal{A}_2,\psi_2,\varphi_2,\omega_2) = (\mathcal{A}_1\star \mathcal{A}_2, \psi_1*\psi_2, \varphi_1\prescript{}{\psi_1}{*}^{}_{\psi_2}\varphi_2, \omega_1\prescript{\psi_1}{\varphi_1}{\circledast}_{\varphi_2}^{\psi_2}\omega_2).
\end{equation*}
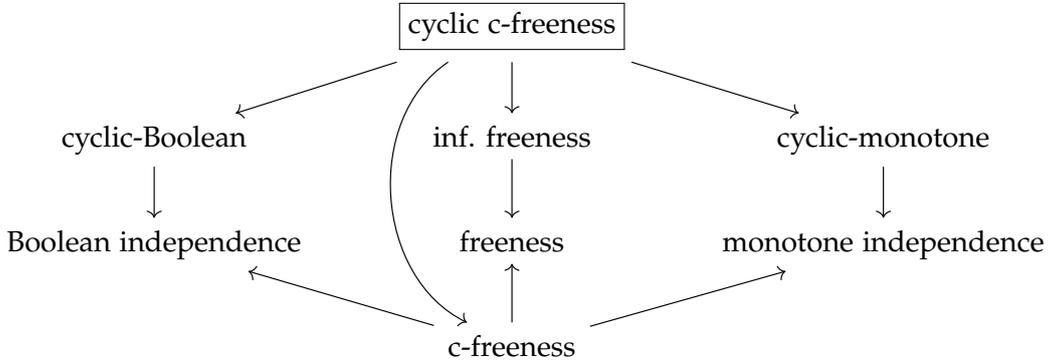
\begin{figure}[!h]
\begin{tikzcd}[ampersand replacement=\&]
		\& {\textrm{\framebox{cyclic c-freeness}}} \\
		{\textrm{cyclic-Boolean}} \& {\textrm{inf. freeness}} \& {\textrm{cyclic-monotone}} \\
		{\textrm{Boolean independence}} \& {\textrm{freeness}} \& {\textrm{monotone independence}} \\
		\& {\textrm{c-freeness}}
		\arrow[from=1-2, to=2-2]
		\arrow[from=1-2, to=2-3]
		\arrow[from=1-2, to=2-1]
		\arrow[from=2-3, to=3-3]
		\arrow[from=2-1, to=3-1]
		\arrow[from=4-2, to=3-1]
		\arrow[from=4-2, to=3-3]
		\arrow[from=4-2, to=3-2]
		\arrow[bend right=60,from=1-2, to=4-2]
		\arrow[from=2-2, to=3-2]
	\end{tikzcd}
	\caption{Each arrow means that the initial independence generalizes the terminal one. \label{Figureone}}
\end{figure}
\item The infinitesimal distribution of ${\bf A}_N$ and $U_N{\bf B}_NU_N^\star$ is determined by cyclic-conditional freeness (see Theorem~\ref{th:cycliccfreeness}). More precisely, let us assume the existence of an $\frac{1}{N}$-expansion of $\psi_{{\bf A}_N}$ and  $\psi_{{\bf B}_N}$, that is, for any
polynomial $P$ in non-commuting variables
$$
\psi_{{\bf A}_N}(P) = \psi_{{\bf A}}(P) + \frac{1}{N}\omega_{\bf A}(P) + o(1/N), \ \ \psi_{{\bf B}_N}(P) = \psi_{{\bf B}_N}(P) + \frac{1}{N}\omega_{\bf B}(P) + o(1/N),
$$  and the convergence of $\varphi^{v_N}_{{\bf A}_N}$ and  $\varphi^{v_N}_{{\bf B}_N}$, that is,
for any
polynomial $P$ in non-commuting variables
$$
\varphi^{v_N}_{{\bf A}_N}(P) = \varphi_{\bf A}(P) +o(1), \ \ \varphi^{v_N}_{{\bf B}_N}(P) = \varphi_{\bf B}(P) + o(1).
$$We prove the existence, for any
polynomial $P$ in two non-commuting variables, of the following $\frac{1}{N}$-expansion $$
\mathbb{E}\left[\frac{1}{N}\mathrm{Tr}[P({\bf A}_N,U_N{\bf B}_NU_N^{\star})] \right]= \psi_{\bf A} * \psi_{\bf B}(P)  + \frac{1}{N} {\omega_{\bf A}}\prescript{\varphi_{\bf A}\!}{\psi_{\bf A}\!}{\circledast}^{\varphi_{\bf B}}_{\psi_{\bf B}} \omega_{\bf B}(P)+ o(1/N),
$$
where $\prescript{\varphi_{A}\!}{\psi_{A}\!}{\circledast}^{\varphi_{B}}_{\psi_{B}}$ denotes the cyclic c-free product of tracial linear functionals introduced above.
Note that the above expansion holds in expectation, but not almost surely. In fact, we show that it remains a random error of order $\frac{1}{N}$ which is asymptotically Gaussian (see Proposition~\ref{prop:gaussian}).
\item
In \cite{hasebe2010new}, Hasebe introduced an associative three-state independence $\leftthreetimes$ named \emph{indented independence}
defined by the product of three unital linear functionals
$$
(\varphi_1,\psi_1,\theta_1) \leftthreetimes (\varphi_2,\psi_2,\theta_2) := (\varphi_1\prescript{}{\theta_1\!\!}{*}^{}_{\psi_2}\varphi_2, \psi_1\prescript{}{\theta_1\!\!}{*}^{}_{\psi_2} \psi_2, \theta_1\prescript{}{\theta_1\!\!}{*}^{}_{\psi_2}\theta_2).
$$
He also introduced an associative two-state independence $\leftthreetimes$ named \emph{ordered freeness} ($o$-freeness for short)
defined by the product of two unital linear functionals
$$
(\varphi_1,\psi_1) \leftthreetimes (\varphi_2,\psi_2) := (\varphi_1\prescript{}{\psi_1\!\!}{*}^{}_{\varphi_2}\varphi_2, \psi_1\prescript{}{\psi_1\!\!}{*}^{}_{\varphi_2}\psi_2).
$$
Elaborating on the Vortex model described above, we are able to build two other models displaying asymptotically ordered freeness and indented independence (see Sections~\ref{sec:matrixmodelordered}  and ~\ref{sec:matrixmodelindented}). In the last model, we allow both ${\bf A}_N$ and ${\bf B}_N$ to be randomly rotated by unitaries conditionned to leave invariant fixed intersecting two dimensional subspaces.
\end{enumerate}

\usetikzlibrary{decorations.pathmorphing}

\subsection{Notations} We introduce a few items of notations that will be used through the paper.
Let $N\geq 1$ be a natural number. We will use the following notations:
\begin{itemize}
\item $\mathcal{M}_{N}(\mathbb{C})$ is the complex vector space of all complex $N$-dimensional matrices,
\item $\langle-,-\rangle$ is the canonical hermitian scalar product on $\mathbb{C}^{N}$:
$$
\langle (x_1,\ldots,x_N), (y_1,\ldots,y_N)\rangle = \sum_{i=1}^N \bar{x}_i y_i
$$

\item We let $U(N)=\{U\in \mathcal{M}_{N}(\mathbb{C}):U_NU_N^{\star}=I_N\}$ be the group of unitary matrices of size $N\times N$, and we denote by
$$\mathrm{Stab}(v)=\{M\in \mathcal{M}_{N}(\mathbb{C}):Mv=v\}$$ the algebra of matrices of dimensions $N\times N$ leaving invariant $v$. For any integer $k<N$ and orthonormal family $v_1,\ldots,v_k\in \mathbb{C}^N$, the group $U(N)\cap \mathrm{Stab}(v_1)\cap \cdots \cap \mathrm{Stab}(v_k)$ is isomorphic to the group of unitary matrices $U({N-k})$ where $k$ is the dimension of the vector space generated by $v_1,\ldots,v_k$ and is therefore equipped with a Haar measure. When we speak about a uniform unitary matrix leaving $v_1,\cdots,v_k$ invariant, we mean a random matrix of $U(N)\cap \mathrm{Stab}(v_1)\cap \cdots \cap \mathrm{Stab}(v_k)$ whose distribution is the Haar measure.

\item We use $\mathrm{tr}_N$ for the normalized trace on $\Mm$ and Tr$_N$ for the non-normalized trace,
$$
\mathrm{tr}_N(A)=\frac{1}{N}\sum_{i=1}^NA_{i,i},\quad \mathrm{Tr}_N(A)=\sum_{i=1}^NA_{i,i},\quad A = (A_{i,j}:1 \leq i,j \leq N)\in \Mm.
$$

\item We use $\mathbb{C}\langle X_k:k\in K \rangle$ for polynomials with complex coefficients on non-commutative indeterminates $X_k,~k\in K$.
\end{itemize}
\subsection{Organization of this paper}
\begin{itemize}
\item In Section \ref{sec:preliminaries}, we recall the relevant definitions of freeness, conditional freeness and the lesser known indented and ordered independences.

\item In Section \ref{Sec:cfreeness} contains our first original contribution,that is the introduction cyclic-conditional freeness. Besides, we explain how one recovers cyclic-Boolean and cyclic-monotone independences but also infinitesimal freeness from cyclic-conditional freeness thereby building a triptych analogous to Boolean-monotone-free independences.

\item In Section \ref{sec:asymptotic}, we prove asymptotic conditional freeness and asymptotic cyclic conditional freeness for the Vortex model.

\item In Section \ref{sec:amalgamation}, we discuss extension of the results proved in Section \ref{sec:asymptotic} in the amalgamated setting.
\end{itemize}

\section{Preliminaries}
\label{sec:preliminaries}
We recall basic definitions pertaining to freeness and conditional freeness. The reader is directed to the monograph \cite{nica2006lectures} for a broad account on freeness and the articles \cite{bozejko1991independent,bozejko1996convolution} introducing conditional freeness.
\subsection{Freeness and conditional freeness}\label{Def:freeness}

\begin{definition}
\begin{enumerate}
    \item A non-commutative probability space is an unital algebra $\mathcal{A}$ equipped with an unital functional $\psi:\mathcal{A}\to \mathbb{C}$.
    \item A linear functional $\psi:\mathcal{A}\to\mathbb{C}$ will be called \emph{tracial} if $\psi(ab)=\psi(ba)$, $a,b\in\mathcal{A}$.
    \item A random variable is an element of $\mathcal{A}$. Given a set $\mathbf{A}=\{a_k:k\in K\}$ of random variables in $\mathcal{A}$, the distribution of $\mathbf{A}$ is the collection of all numbers
    $$
    \psi(a_{k_1}\cdots a_{k_n}),~k_1,\ldots,k_n \in K.
    $$
    This is nicely encoded in the linear functional $\psi_{\mathbf{A}}:\mathbb{C}\langle X_k:k\in K \rangle \to \mathbb{C}$ given by
    $$\psi_{\mathbf{A}}(P):=\psi(P({\mathbf{A}})).$$
\end{enumerate}
\end{definition}
Let $\mathcal{A}$ be an unital algebra. We let $(\mathcal{A}_i:i\in I)$ be a family of unital sub-algebras of $\mathcal{A}$. We say that a finite sequence of elements $(a_1,\ldots,a_n)$ , $a_j\in \cup_{i\in I}\mathcal{A}_i,~1\leq j \leq n$ is:
\begin{itemize}
\item \emph{alternating} if $a_1\in \mathcal{A}_{i_1},\ldots,a_n\in \mathcal{A}_{i_n}$ are such that $i_1\neq i_2\cdots\neq i_n$;
\item \emph{cyclically alternating} if $a_1\in \mathcal{A}_{i_1},\ldots,a_n\in \mathcal{A}_{i_n}$ are such that $i_1\neq i_2\neq\cdots\neq i_n$ and $i_n\neq i_1$;
\item \emph{centered} with respect to a linear functional $\psi:\mathcal{A}\to \mathbb{C}$ if $\psi(a_1)=\cdots=\psi(a_n)=0$.
\end{itemize}

\begin{definition}[Free independence] With the notations introduced so far, we say that $(\mathcal{A}_i\subset \mathcal{A}:i\in I)$ are \emph{free} with respect to a unital linear functional $\psi:\mathcal{A}\to \mathbb{C}$ if for any sequence $ (a_1,\ldots,a_n)$ of $\cup_{i\in I}\mathcal{A}_i$ which is alternating and centered with respect to $\psi$, we have
$$
\psi(a_1\cdots a_n) = 0
$$
Given a set of random variables $ \{a_1,\ldots,a_k\}$, we say that they are mutually free if the algebras they each generate are free.
\end{definition}
We use $\star$ for the free product of unital algebra with identification of units: $\mathcal{A}_1 \star \mathcal{A}_2$ is the unital algebra whose elements are alternated words on $\mathcal{A}_1$ and $\mathcal{A}_2$ with product given by concatenation followed by reduction; two neighbouring letters in either $\mathcal{A}_1$ or $\mathcal{A}_2$ are replaced by their product.

Given two probability spaces $(\mathcal{A}_1,\psi_1)$ and $(\mathcal{A}_2,\psi_2)$, we denote by $\psi_1*\psi_2$ the unique linear functional on $\mathcal{A}_1\star \mathcal{A}_2$ extending $\psi_1$ and $\psi_2$ and such that $\mathcal{A}_1 \subset \mathcal{A}_1\star\mathcal{A}_2$ and $\mathcal{A}_2 \subset \mathcal{A}_1\star\mathcal{A}_2$ are two free subalgebras.   In particular, two sets of random variables $\mathbf{A},\mathbf{B}$ are free with respect to $\psi$ if and only if	\begin{equation}
\psi_{\mathbf{A},\mathbf{B}}=	\psi_{\mathbf{A}}*\varphi_{\mathbf{B}}.\label{eq:freeness}
	\end{equation}
	Given two sequences of sets of random variables $\mathbf{A}_N,\mathbf{B}_N$, we said that they are \emph{asymptotically free} if \eqref{eq:freeness} holds up to $o(1)$ as $N\to \infty$.
\begin{definition}[c-free independence] With the notations introduced so far, we say that $(\mathcal{A}_i\subset \mathcal{A}:i\in I)$ are
 \emph{conditionally free} (or c-free for short) with respect to a pair of unital linear functionals $\psi,\varphi:\mathcal{A}\to \mathbb{C}$ if for any sequence $ (a_1,\ldots,a_n)$ of  $\cup_{i\in I}\mathcal{A}_i$ which is alternating and centred with respect to $\psi$, we have
$$
\psi(a_1 \cdots a_n) = 0 \textrm{ and } \varphi(a_1\cdots a_n) = \varphi(a_1) \cdots \varphi(a_n).
$$
\end{definition}

Given two unital algebras $\mathcal{A}_1$ and $\mathcal{A}_2$ and two pairs of unital linear functionals $\psi_1,\varphi_1 :\mathcal{A}_1\to \mathbb{C}$ and $\psi_2,\varphi_2 :\mathcal{A}_2\to \mathbb{C}$, we denote by $$(\psi_1,\varphi_1)*(\psi_2,\varphi_2)=(\psi_1*\psi_2,\varphi_1\prescript{}{\psi_1}{*}^{}_{\psi_2}\varphi_2)$$ the unique pair of linear functionals on the free product $\mathcal{A}_1\star\mathcal{A}_2$ extending $(\psi_1,\varphi_1)$ and $(\psi_2,\varphi_2)$ and such that $\mathcal{A}_1 \subset \mathcal{A}_1 \star \mathcal{A}_2$ and $\mathcal{A}_2 \subset \mathcal{A}_1 \star \mathcal{A}_2$ are conditionally free with respect to $(\psi_1*\psi_2,\varphi_1\prescript{}{\psi_1}{*}^{}_{\psi_2}\varphi_2)$. While the free product $\psi_1*\psi_2$ does not depend on $\varphi_1$ and $\varphi_2$, the c-free product $\varphi_1\prescript{}{\psi_1}{*}^{}_{\psi_2}\varphi_2$ depends on the four linear functionals.

Two sets of random variables $\mathbf{A},\mathbf{B}$ are $c$-free with respect to $(\psi,\varphi)$ whenever	\begin{equation}
(\psi_{\mathbf{A},\mathbf{B}},\varphi_{\mathbf{A},\mathbf{B}})=	(\psi_{\mathbf{A}},\varphi_{\mathbf{A}})*(\psi_{\mathbf{B}},\varphi_{\mathbf{B}}).\label{eq:cfreeness}
	\end{equation}
	Given two sequences of sets of random variables $\mathbf{A}_N,\mathbf{B}_N$, we said that they are \emph{asymptotically $c$-free} if \eqref{eq:cfreeness} holds up to $o(1)$ as $N\to \infty$.
\subsection{Indented independence and ordered freeness}
\label{sec:indented}
Hasebe's indented independence was introduced in \cite{hasebe2010new} and unifies many independences: free, monotone, anti-monotone, Boolean, conditionnaly free, conditionnaly monotone and conditionally anti-monotone independences.

Indented independence is cast as an associative product on triple of states: it preserves positivity. We shall however introduce it as operating on triple of linear functionals.


\begin{definition}[Indented independence]\begin{enumerate}
    \item Let $\mathcal{A}_1$ and $\mathcal{A}_2$ be unital algebras. Let $(\varphi_1,\psi_1,\theta_1)$ and $(\varphi_2,\psi_2,\theta_2)$ be two triples of unital linear functionals on $\mathcal{A}_1$, respectively on $\mathcal{A}_2$.
The \emph{indented product} is defined by the following triple of unital linear functionals on $\mathcal{A}_1\star\mathcal{A}_2$
$$
	(\varphi_1,\psi_1,\theta_1)\leftthreetimes (\varphi_2,\psi_2,\theta_2):=(\varphi_1 \prescript{}{\theta_1}{*}^{}_{\psi_2}\varphi_2,\psi_1 \prescript{}{\theta_1}{*}^{}_{\psi_2} \psi_2, \theta_1 \prescript{}{\theta_1}{*}^{}_{\psi_2}\theta_2).$$
	\item Let $\mathcal{A}$ be a unital algebra equipped with a triple 	$(\varphi,\psi,\theta)$ of unital linear functionals. Two families of random variables $\mathbf{A}$ and $\mathbf{B}$ are said to be \emph{indented independent} with respect to $(\varphi,\psi,\theta)$ whenever their distribution is given by the indented product
\begin{equation}(\varphi_{\mathbf{A},\mathbf{B}},\psi_{\mathbf{A},\mathbf{B}},\theta_{\mathbf{A},\mathbf{B}})=	(\varphi_{\mathbf{A}},\psi_{\mathbf{A}},\theta_{\mathbf{A}})\leftthreetimes (\varphi_{\mathbf{B}},\psi_{\mathbf{B}},\theta_{\mathbf{B}}).\label{eq:indentedind}\end{equation}
		\item Given two sequences of sets of random variables $\mathbf{A}_N,\mathbf{B}_N$, we said that they are \emph{asymptotically indented independent} if \eqref{eq:indentedind} holds up to $o(1)$ as $N\to \infty$.
\end{enumerate}
\end{definition}
The particular case of $(\varphi_1,\psi_1,\psi_1)\leftthreetimes (\varphi_2,\varphi_2,\psi_2)$ yields an associative product on pairs of unital linear functionals which can be given below as a separate definition, following \cite{hasebe2010new}.
\begin{definition}[Ordered freeness]
\begin{enumerate}
    \item Let $\mathcal{A}$ be an unital algebra and $\mathcal{A}_1,\mathcal{A}_1 \subset$ be unital subalgebras of $\mathcal{A}$ endowed with two linear functionals $(\varphi,\psi)$. We say that $\mathcal{A}_1$ and $\mathcal{A}_2$ are $o$-free independent if for any alternated word $x_1\cdots x_n$, $x_i \in \mathcal{A}_{i_j}$ such that 
    $$
    \varphi(x_{i_j})=0 \textrm{ if } i_j=2,\quad \psi(x_{i_j})=0 \textrm{ if } i_j=1
    $$
    one has $\varphi(x_1\cdots x_n) = \psi(x_1\cdots x_n) = 0$.
    \item Let $\mathcal{A}_1$ and $\mathcal{A}_2$ be unital algebras. Let $(\varphi_1,\psi_1)$ and $(\varphi_2,\psi_2)$ be two pairs of unital linear functionals on $\mathcal{A}_1$, respectively on $\mathcal{A}_2$.
The \emph{ordered product} is defined by the following pair of unital linear functionals on $\mathcal{A}_1\star\mathcal{A}_2$
$$
	 (\varphi_1,\psi_1) \leftthreetimes (\varphi_2,\psi_2) := (\varphi_1\prescript{}{\psi_1}{*}^{}_{\varphi_2}{\varphi_2}, \psi_1\prescript{}{\psi_1}{*}^{}_{\varphi_2}\psi_2).$$
	\item Let $\mathcal{A}$ be a unital algebra equipped with a pair 	$(\varphi,\psi)$ of unital linear functionals. Two families of random variables $\mathbf{A}$ and $\mathbf{B}$ are said to be \emph{ordered free} (or o-free for short) with respect to $(\varphi,\psi)$ whenever their distribution is given by the ordered product
	\begin{equation}
(\varphi_{\mathbf{A},\mathbf{B}},\psi_{\mathbf{A},\mathbf{B}})=	(\varphi_{\mathbf{A}},\psi_{\mathbf{A}})\leftthreetimes (\varphi_{\mathbf{B}},\psi_{\mathbf{B}}).\label{eq:ofreeness}
	\end{equation}
	\item Given two sequences of sets of random variables $\mathbf{A}_N,\mathbf{B}_N$, we said that they are \emph{asymptotically o-free} if \eqref{eq:ofreeness} holds up to $o(1)$ as $N\to \infty$.
\end{enumerate}
\end{definition}
\subsection{Freeness up to order $O(N^{-2})$ of random matrices}

Let ${\bf A}_N = \{A_N^k: k \in K\}$ be a  family of $N\times N$ matrices. The non-commutative distribution of ${\bf A}_N$ is the linear functional $\psi_{{\bf A}_N}:\mathbb{C}\langle X_k:k\in K\rangle \to \mathbb{C}$ given by
$$\psi_{{\bf A}_N}(P)=\trn(P({\bf A}_N)).$$
Asymptotic freeness of Voiculescu for unitarily invariant matrices is known to hold up to order $O(N^{-2})$. We refer to \cite{cebron2022freeness} for  some bibliographic notes about the following theorem and two different proofs.
\begin{theorem}\label{th:asymptotic_freeness}
    Let $R>0$ and, for each $N\geq 1$, let $U_N$ be a Haar distributed unitary random matrix. Then, for any $P\in \mathbb{C}\langle X_k,Y_k:k\in K\rangle$, we have
    $$
    \mathbb{E}\left[\trn(P({\bf A}_N,U_N{\bf B}_NU_N^{\star}))\right] = \psi_{{\bf A}_N} * \psi_{{\bf B}_N}(P)+O(N^{-2})
    $$
    uniformly for the the choice of any sequences ${\bf A}_N = \{A_N^k: k \in K\}$ and ${\bf B}_N = \{B^k_N:~k \in K\}$ of $N\times N$ matrices ($N\geq 1$) bounded in operator norm by $R$.
\end{theorem}
The previous result can also be seen as a consequence of the following proposition, which relies on~\cite{collins2003moments}.
\begin{proposition}[Proposition 5.11 \cite{curran2011asymptotic}]
    Let $R>0$ and, for each $N\geq 1$, let $U_N$ be a Haar distributed unitary random matrix. Then, for any $P\in \mathbb{C}\langle X,X^{-1},Y_k:k\in K\rangle$, we have
    $$
    \mathbb{E}\left[\trn(P(U_N,{\bf M}_N))\right] = \mathbb{E}[\psi_{U_N}]* \psi_{{\bf M}_N}(P)+O(N^{-2})
    $$
    uniformly for the choice of any sequences ${\bf M}_N = \{M^k_N:~k \in K\}$ of $N\times N$ matrices ($N\geq 1$) bounded in operator norm by $R$.
\end{proposition}
In particular, given $n_1,\ldots,n_\ell\in \mathbb{Z}\setminus \{0\}$, we have
$$\mathbb{E}\left[\trn\left(M_N(1)U_N^{n_1}M_N(2)U_N^{n_2}\cdots M_N(\ell)U_N^{n_\ell}\right)\right]=O(N^{-2})$$
for the choice of any sequences of $\ell$-tuples $(M_N(1),\ldots,M_N(\ell))$ of $N\times N$ matrices $(N\geq 1$) bounded in operator norm by $R$ and centred with respect to $\trn$. It is possible to relax the centering of $M_N(1),\ldots,M_N(\ell)$ as follows.

\begin{proposition}\label{prop:freenesswithunitaries}
Let $R>0$, $n_1,\ldots,n_\ell\in \mathbb{Z}\setminus \{0\}$ and, for each $N\geq 1$, let $U_N$ be a Haar distributed unitary random matrix. Then
$$\mathbb{E}\left[\trn\left(M_N(1)U_N^{n_1}M_N(2)U_N^{n_2}\cdots M_N(\ell)U_N^{n_\ell}\right)\right]=O(N^{-2})$$
uniformly for the choice of any sequences of $\ell$-tuples  $(M_N(1),\ldots,M_N(\ell))$ of $N\times N$ matrices ($N\geq 1$) bounded in operator norm by $R$ and such that, for all $1\leq k\leq \ell$ and $1\leq N$,
$$\trn(M_N(k))\leq R/N.$$
\end{proposition}
\begin{proof}Setting $P=Y_1X^{n_1}Y_2X^{n_1}\cdots Y_\ell X^{n_\ell}$, we have
\begin{align*}
    \mathbb{E}\left[\trn\left(M_N(1)U_N^{n_1}M_N(2)U_N^{n_2}\cdots M_N(\ell)U_N^{n_\ell}\right)\right]&=\mathbb{E}\left[\trn(P(U_N,M_N)\right]\\
    &=\mathbb{E}[\psi_{U_N}]* \psi_{M_N}(P)+O(N^{-2})\\
    &=O(N^{-2}),
\end{align*}
where we used Lemma~\ref{lem:almostcentered} in the last line, because $P$ is alternated but not centered with respect to $\mathbb{E}[\psi_{U_N}]* \psi_{M_N}$.\end{proof}
Though stated for a single Haar unitary matrix $U_N$, Proposition \ref{prop:freenesswithunitaries} applies in much more general situations:

\begin{enumerate}
    \item First, if for each $N \geq 1$, one is given a family of independent Haar unitary matrices $(U_N(s))_{s \in S}$ then, under the same hypothesis of Proposition \ref{prop:freenesswithunitaries} and for any $s_1,\ldots,s_\ell \in S$:
    \begin{equation*}
        \mathbb{E}[\trn(M_N(1)U_N^{n_1}(\ell_1)M_N(2)U_N^{n_2}(s_2)\cdots M_N(\ell)U_N^{s_\ell}] = O(N^{-2}),
    \end{equation*}
    which extends the validity of \cite[Lemma 4.3.2]{hiai2000semicircle} to non-centred matrices.
    \item Secondly, the ensemble of Haar independent unitary matrices can be replaced by an ensemble of quantum free unitary matrices $(\mathcal{U}_N(s))_{s \in S}$; where each $\mathcal{U}_N(s)$ is Haar distributed in the free unitary group $U_N^+$ (see \cite{curran2011asymptotic} for definitions).
\end{enumerate}

\begin{lemma}\label{lem:almostcentered}Let $T\subset \mathbb{R}$ be an index set such that $0$ is an accumulation point of $T$.
Let $\mathcal A$ be a unital complex algebra endowed with a family of linear functionals $\psi^t:\mathcal{A}\to \mathbb{C}$ ($t\in T$) which is pointwise bounded. Let $\mathcal{A}_1$ and $\mathcal{A}_2$ be two algebras freely independent with respect to $\psi^t$ for any $t\in T$. For any \emph{cyclically alternating} sequence $a_1,b_1,\ldots,a_n,b_n$ in $\mathcal{A}_1\cup \mathcal{A}_2$ such that
\begin{equation*}
\label{eqn:almostcentered}
\psi^t(a_i)=O(t)\ \text{ and }\ \psi^t(b_i) = O(t)
\end{equation*}
then, one has
$$
\psi^t(a_1b_1\cdots a_nb_n) = O(t^2).
$$
\end{lemma}
\begin{proof} We appeal to the Kreweras complement of a non-crossing partition and to the free cumulants (see \cite{nica2006lectures} for details). We recall succinctly the definition of the Kreweras complement. Let $\pi$ be a non-crossing partition in $\mathrm{NC}(n)$.
Set $[n]^{\prime} = \{1 < 1^{\prime} < \cdots n < n^{\prime}\}$. Given two partitions $\alpha$ and $\beta$ of $\{1,\ldots,n\}$, respectively $\{1^{\prime},\ldots,n^{\prime}\}$, we denote by $\alpha \sqcup \beta$ the unique partition of $[n]^{\prime}$ which restrict to $\alpha$ on $\{1,\ldots,n\}$ and to $\beta$ on $\{1^{\prime},\ldots,n^{\prime}\}$. The Kreweras partition if the largest partition (for the inverse refinement order) such that $\pi \sqcup {\sf K}(\pi)$ is non-crossing.
Because (see \cite[Theorem 14.4.]{nica2006lectures})
$$
\psi^t(a_1b_1\cdots a_nb_n) = \sum_{\pi \in \mathrm{NC}(p)} \kappa^t_{\pi}(a_1,\cdots)\psi^t_{{\sf K}(\pi)}(b_1,\cdots),
$$
it suffices to notice that $\pi \sqcup K(\pi)$ as at least two singletons to conclude. In fact, let $V_1,\ldots,V_n$ the interval blocks in $\pi$. The number of singletons in $\pi \sqcup {\sf K}(\pi)$ is at least $\sum_{j=1}^n(|V_i|-1) + \mathrm{singletons}(\pi)$.
If the above expression is equal to one, it means that
\begin{enumerate}
    \item singletons($\pi$)$=1$ and there is only one interval block (which is the only singleton). In that case, $\pi$ is irreducible and $\{1\} \in {\sf K}(\pi)$.
    \item $\sum_{j=1}^{n}(|V_i|-1)=1$ and $\pi$ contains no singletons : it contains an unique interval block with size 2. Again, in that case $\pi$ is irreducible and thus $\{1\} \in \sf{K}(\pi)$.
\end{enumerate}
\end{proof}


\subsection{Concentration and fluctuations}Asymptotic freeness of Theorem~\ref{th:asymptotic_freeness} can be turned into an almost sure asymptotic freeness. One way to prove it is to use the following concentration of Haar measure on the unitary group $U(N)$.\label{sec:fluctutations}

\begin{theorem}[Corollary 17 of \cite{meckes2013spectral}]
\label{th:concentration_unitary}Let $f$ be a continuous real-valued function on $U(N)$ which, for some constant $C$ and all $U,V\in U(N)$ satisfies
$$|f(U)-f(V)|\leq C\sqrt{\Tr((U-V)(U-V)^\star)}.$$
Let $U_N$ be  Haar distributed on $U(N)$. Then we have for all $\delta>0$,
$$\mathbb{P}\left[\left|f(U_N)-\mathbb{E}[f(U_N)]\right|\geq \delta\right]\leq 2e^{-\frac{N\delta^2}{12C^2}}.$$
\end{theorem}
This concentration result can be used together with the Borel–Cantelli lemma to show that the normalized trace has the same behavior as $N\to \infty$ with or without taking the expectation. The difference is a Gaussian error of order $\frac{1}{N}$.

Let us present now the theory of second order free probability which can be used in order to describe this Gaussian error. In the sequel, we will denote by $k_r(X_1,\ldots,X_r),~X_i \in {\bf X}$ the $r^{\text{th}}$ classical cumulants of a family ${\bf X}$ of random variables.  We recall the basic definitions of second-order freeness.
\begin{definition}
Let ${\bf A}_N$ be an ensemble of $N\times N$ random matrices. We say that it has a \emph{second-order limit distribution} if the following limits exist for any polynomials $P_i$ in non-commuting variables:
\begin{align*}
    \psi_{\bf A}(P_1) :=& \lim_{N\to \infty} \mathbb{E}[\trn(P_1({\bf A}_N))],\\
        \psi_{{\bf A}}^{(2)}(P_1,P_2):=& \lim_{N\to \infty} k_2(\mathrm{Tr}_N(P_1({\bf A}_N)),\mathrm{Tr}_N(P_2({\bf A}_N))),\\
        0=& \lim_{N\to \infty}k_r(\mathrm{Tr}_N(P_1({\bf A}_N)),\ldots,\mathrm{Tr}_N(P_r({\bf A}_N))\ \ \text{with}\ \ r > 2.
\end{align*}
\end{definition}
Note that if the ensemble $\bfA_N$ is an ensemble of deterministic matrices, then ${\bfA}_N$ has second-order limit distributions if and only if $\psi_{\bf A _N}$ converges point-wise, since all cumulants of ${\bfA}_N$ of order greater than two vanishes.
\begin{definition}
\begin{enumerate}
    \item A \emph{second order non-commutative probability space} $(\mathcal{A},\psi_1,\psi_2)$ consists of a unital algebra $\mathcal{A}$, a tracial linear functional $\psi:\mathcal{A}\to \mathbb{C}$ and a bilinear functional 
    $$
    \psi^{(2)}:\mathcal{A}\to\mathbb{C}
    $$
    which is tracial in both arguments and which satisfies 
    $$
    \psi_2(a,1) = 0 = \psi_2(1,b),~a,b \in \mathcal{A}
    $$
    \item We say that two subalgebras $\mathcal{A}_1,\mathcal{A}_2 \subset \mathcal{A}$ are \emph{free of second order} with respect to $(\psi,\psi^{(2)})$ whenever $\mathcal{A}_1$ and $\mathcal{A}_2$ are free with respect to $\psi$ and the following condition involving $\psi^{(2)}$ is satisfied. For $n,m \geq 1$ integers and tuples $(a_1,\ldots,a_n)$, $(b_m,\ldots,b_1)$ from $\mathcal{A}$ such that both are cyclically alternating and centered with respect to $\psi$, one has
        $$
        \psi^{(2)}(a_1\cdots a_n,b_n\cdots b_1)= \delta_{m,n}\sum_{k=0}^{n-1} \psi(a_1b_{1+k})\cdots \psi(a_nb_{n+k}).
        $$
where the indices of the $b_i$ are interpreted modulo $n$.
\end{enumerate}
\end{definition}

If the algebra $\mathcal{A}_1$ and $\mathcal{A}_2$ are free of second order, then $\psi^{(2)}$ restricted (on both of its arguments) to the algebra generated by $\mathcal{A}_1$ and $\mathcal{A}_2$ is determined by the restrictions of $\psi$ and $\psi^{(2)}$ to each algebra. We may thus adopt a notation that will be turn useful later. Given two second-order probability spaces $(\psi_{\bfA},0)$ and $(\psi_{\bfB},0)$ (where $0$ is the trivial bilinear functional), we denote by $ (\psi_{\bfA}*\psi_{\bfB})^{(2)} $ the unique bilinear functional on $\mathcal{A}_1\star\mathcal{A}_2$ extending $(\psi_{\bfA},0)$ and $(\psi_{\bfB},0)$  and such that  $\mathcal{A}_1$ and $\mathcal{A}_2 $ are free of second order in the second order probability space $$\left(\mathcal{A}_1\star\mathcal{A}_2, \psi_{\bfA}* \psi_{\bfB},(\psi_{\bfA}*\psi_{\bfB})^{(2)}\right).$$

We recall here two results which are the second-order counterparts of asymptotic freeness of randomly rotated matrices.
\begin{proposition}[Theorem 3.12 of \cite{mingo2007second}]
\label{prop:asymptoticsecondorder}
For each $N\geq 1$, let $\bfA_N$ be an ensemble of deterministic matrices. We suppose $\bfA_N$ has a second order limit distribution. for each $N\geq 1$, let $U_N$ be a Haar distributed unitary random matrix. Then $\bfA_N$ and $U_N$ are asymptotically free of second order.
\end{proposition}

\begin{theorem}
[Theorem 3.15 of \cite{mingo2007second}]
\label{thm:asymptoticsecondrderfreeness}
Let ${\bf A}_N = \{A_N^k: k \in K\}$ and ${\bf B}_N = \{B^k_N:~k \in K\}$ be two sequences of deterministic $N\times N$ matrices ($N\geq 1$) such that $\psi_{\bfA_N}$ (resp. $\psi_{\bfB_N}$) converges point-wise to a linear functional $\psi_{\bfA}$ (resp. $\psi_{\bfB}$). For each $N\geq 1$, let $U_N$ be a Haar distributed unitary random matrix. Then
$\bfA_N$ and $U_N\bfB_NU_N^{\star}$
are asymptotically free of second order.

In particular,
for any $P^{(1)},\ldots,P_{(r)} \in \mathbb{C}\langle X_k,Y_k:k\in K \rangle$, each one being a product of cyclically alternated elements which are centered with respect to $\psi_{\bfA}*\psi_{\bfB}$, we have the convergence in moments of
\begin{equation*}
    \Big(\mathrm{Tr}_N(P^{(1)}(\bfA_N,U_N\bfB_NU_N^{\star})),\ldots,\mathrm{Tr}_N(P^{(r)}(\bfA_N,U_N\bfB_NU_N^{\star}))\Big)
\end{equation*}
to a Gaussian vector $(Z_1,\ldots,Z_r)$ with covariance matrix $\gamma$ given by
\begin{equation*}
\gamma(i,j) =\delta_{\ell(i),\ell(j)} \sum_{i=1}^{\ell(i)} \psi_{\bfA}* \psi_{\bfB}(P^{(i)}_1P_{1+k}^{(j)})\cdots\psi_{\bfA}*\psi_{\bfB}(P^{(i)}_rP_{r+k}^{(j)})
\end{equation*}
where $P^{(i)} = P_1^{(i)}\cdots P_{\ell (i)}^{(i)}$ is the decomposition of $P_i$ into cyclically alternated and centered factors.
\end{theorem}
Denoting by $\psi^{(2)}_{\bfA_N,U_N\bfB_NU_N^{\star}}:\mathbb{C}\langle X_k,Y_k:k\in K \rangle\to\mathbb{C}$ the second-order joint distribution of $\bfA_N$ and $U_N\bfB_NU_N^{\star}$, given by
\begin{equation*}
    \psi_{\bfA_N,U_N\bfB_NU_N^{\star}}^{(2)}(P,Q)=k_2(\textrm{Tr}_N(P(\bfA_N)),\textrm{Tr}_N(Q(U_N\bfB_NU_N^{\star}))=\textrm{Cov}(\textrm{Tr}_N(P(\bfA_N)),\textrm{Tr}_N(Q(U_N\bfB_NU_N^{\star}))),
\end{equation*}
the above theorem can be restated as follows:
\begin{equation*}
    \left(\mathbb{E}\left[\psi_{\bfA_N,U_N\bfB_NU_N^{\star}}\right],\psi^{(2)}_{\bfA_N,U_N\bfB_NU_N^{\star}}\right) = \left(\psi_{\bfA}* \psi_{\bfB},(\psi_{\bfA}*\psi_{\bfB})^{(2)}\right)+ o(1),
\end{equation*}
and, for any $r\geq 3$,
\begin{equation*}
    k_r\Big(\mathrm{Tr}_N(P(\bfA_N,U_N\bfB_NU_N^{\star}),\ldots,\mathrm{Tr}_N(Q(\bfA_N,U_N\bfB_NU_N^{\star}))\Big) =o(1).
\end{equation*}
It is possible to weaken the hypothesis of Theorem \ref{thm:asymptoticsecondrderfreeness} about convergence of the first order distributions of the ensemble $\bfA_N$ and $\bfB_N$ and state asymptotic second-order freeness of randomly rotated ensembles of deterministic matrices akin to \ref{th:asymptotic_freeness}.
The following proposition is a mere reformulation of Theorem 3.15 in \cite{mingo2007second}.

\begin{proposition}Let $R>0$ and, for each $N\geq 1$, let $U_N$ be a Haar distributed unitary random matrix. Then, for $P_1,\ldots,P_r\in \mathbb{C}\langle X_k,Y_k:k\in K\rangle$ (with $r\geq 3$), we have
$$\mathbb{E}\left[\trn(P_1(\bfA_N,U_N\bfB_NU^{\star}_N)))\right]=\psi_{\bfA_N}*\psi_{\bfB_N}(P_1) + O(N^{-2}),$$
$$k_2\Big(\Trn(P_1(\bfA_N,U_N\bfB_NU^{\star}_N)),\Trn(P_2(\bfA_N,U_N\bfB_NU^{\star}_N))\Big) =(\psi_{\bfA_N}*\psi_{\bfB_N})^{(2)}(P_1,P_2)+O(N^{-2}),$$
and, 
    $$
k_r\Big(\Trn(P_1(\bfA_N,U_N\bfB_NU^{\star}_N)), \ldots,\Trn(P_r(\bfA_N,U_N\bfB_NU^{\star}_N))\Big) =o(1)
$$
uniformly for the the choice of any sequences ${\bf A}_N = \{A_N^k: k \in K\}$ and ${\bf B}_N = \{B^k_N:~k \in K\}$ of $N\times N$ matrices ($N\geq 1$) bounded in operator norm by $R$.
\end{proposition}

\begin{proof}[Sketch of proof]We refer to \cite[Theorem 3.15]{mingo2007second} for the details of the proof. It ultimately relies on the fact that geodesic deviations in the symmetric group occur at even orders:
\begin{equation*}
    d(\mathrm{id},\tau) + d(\tau,\sigma) - d(\mathrm{id},\sigma) \in 2\mathbb{N}.
\end{equation*}
and on the following form of Theorem 3.12 in \cite{mingo2007second} (we use the notations introduced by the authors without recalling them here),
\begin{align*}
    &k_2(\mathrm{Tr}_N(A_1U_N^{\varepsilon_1}A_2U_N^{\varepsilon_2} \cdots A_k U_N^{\varepsilon_k},~\mathrm{Tr}_N(A_{k+1}U_N^{\varepsilon_{k+1}}A_{k+2}U_N^{\varepsilon_{k+2}} \cdots A_{k+q}U_N^{\varepsilon_{k+q}})) \\
    &\hspace{1cm}=\sum_{\mathcal{S}^{\varepsilon}(k,q)}\mu(\tilde{\pi}) [\mathrm{tr}_{(c_k \otimes c_q)\circ \pi}(A_1\otimes\cdots \otimes A_{k+q})] \\
    &\hspace{2cm}+\sum_{\pi_1 \in \mathcal{S}^{(\varepsilon_1)}(k),~\pi_2\in\mathcal{S}^{(\varepsilon_2)}(q)}\mu(\tilde{\pi}_1, \tilde{\pi}_2)[\mathrm{tr}_{(c_k\circ\pi_1^{-1}) \otimes (c_q\circ\tilde{\pi}_2^{-1})}(A_1\otimes\cdots A_{k+q})] + O(\frac{1}{N^2})
\end{align*}
where $\mathrm{tr}_{\sigma}(A_1\otimes A_n)$, $\sigma$ a permutation with length $n$ means that the we compute the trace of the product of the matrices $A_i's$ over each cycle of $\sigma$ and $A_1,\ldots,A_{k+q}$ are deterministic matrices.
\end{proof}

\section{Cyclic-conditional freeness}\label{Sec:cfreeness}
\subsection{Definition}
In this section, we introduce a new independence, called cyclic-conditional freeness, between unital algebras equipped with a triple of linear forms among which one is tracial. In the next section we explain how this independence arises naturally when computing infinitesimal distributions of the Vortex model.

\begin{definition}[Cyclic c-freeness] Let $\mathcal{A}$ be a unital complex algebra equipped with three linear functionals $\psi,\varphi,\omega:\mathcal{A}\to \mathbb{C}$ where $\psi,\varphi$ are unital and $\omega$ is tracial.

A family of unital sub-algebras  $(\mathcal{A}_i\subset \mathcal{A}:i\in I)$ is said to be \emph{cyclically conditionally free} (\emph{cyclically c-free} for short) with respect to $(\psi,\varphi,\omega)$ if
\begin{enumerate}
	\item the family $(\mathcal{A}_i\subset \mathcal{A}:i\in I)$ is conditionally free with respect to $(\psi,\varphi)$: for any sequence $(a_1,\ldots,a_n)$ of  $\cup_{i\in I}\mathcal{A}_i$ which is alternating and centred with respect to $\psi$, we have
$$
\psi(a_1 \cdots a_n) = 0 \textrm{ and } \varphi(a_1\cdots a_n) = \varphi(a_1) \cdots \varphi(a_n);
$$
 \item and, for any sequence $ (a_1,\ldots,a_n)$ of  $\cup_{i\in I}\mathcal{A}_i$ which is \emph{cyclically alternating} (thus $n \geq 2)$ and centered with respect to $\psi$, we have
\begin{equation}
\label{eqn:cdnw}
\omega(a_1\cdots a_n) = \varphi(a_1) \cdots \varphi(a_n).
\end{equation}
\end{enumerate}

\end{definition}
\begin{remark}
\begin{enumerate}
\item In the definition above, nothing about the value of $\omega$ on $1_{\mathcal{A}}$ is prescribed.
\item If $\mathcal{A}_1$ and $\mathcal{A}_2$ are cyclically c-free, the values of each of the three functionals $\psi$, $\varphi$ and $w$ are determined by their restrictions to the sub-algebras $\mathcal{A}_1$ and $\mathcal{A}_2$. This is well-known for $\psi$ and $\varphi$. For $\omega$, if $a_1 \cdots a_n$ is alternated but not cyclically alternated, that if $a_1$ and $a_n$ are in the same sub-algebra, $a_na_1,\ldots,a_{n-1}$ is cyclically alternated. We then proceed with $\psi$ centering this last word before applying formula \eqref{eqn:cdnw}.
\end{enumerate}
\end{remark}
As usual, we will say that two random variables $a$ and $b$ are \emph{cyclically c-free} if the two algebras they generate are conditionally c-free.
Let us consider two unital algebras $\mathcal{A}_1$ and $\mathcal{A}_2$ and two triples of linear functionals $\psi_1,\varphi_1,\omega_1 :\mathcal{A}_1\to \mathbb{C}$ and $\psi_2,\varphi_2,\omega_2 :\mathcal{A}_2\to \mathbb{C}$ such that $\psi_1(1_{\mathcal{A}_1})=\psi_2(1_{\mathcal{A}_2})=1$, $\varphi_1(1_{\mathcal{A}_1})=\varphi_2(1_{\mathcal{A}_2})=1$, $\omega_1(1_{\mathcal{A}_1})=\omega_2(1_{\mathcal{A}_2})$,  and $\omega_1,\omega_2$ are tracial. We denote by $$(\psi_1,\phi_1,\omega_1) \circledast (\psi_2,\phi_2,\omega_2)=(\psi_1*\psi_2,\varphi_1\prescript{}{\psi_1}{*}^{}_{\psi_2}\varphi_2,\omega_1\prescript{\varphi_1}{\psi_1}{\circledast}^{\varphi_2}_{\psi_2}\omega_2 )$$ the unique triple of linear functionals on $\mathcal{A}_1\star\mathcal{A}_2$  extending $(\psi_1,\varphi_1,\omega_1)$ and $(\psi_2,\varphi_2,\omega_2)$, such that $\omega_1\prescript{\varphi_1}{\psi_1}{\circledast}^{\varphi_2}_{\psi_2}\omega_2 $ is tracial, and such that $\mathcal{A}_1\subset\mathcal{A}_1\star\mathcal{A}_2$ and $\mathcal{A}_2 \subset \mathcal{A}_1\star\mathcal{A}_2$ are cyclically c-free with respect to $(\psi_1*\psi_2,\varphi_1\prescript{}{\psi_1}{*}^{}_{\psi_2}\varphi_2,\omega_1\prescript{\varphi_1}{\psi_1}{\circledast}^{\varphi_2}_{\psi_2}\omega_2 )$. Note that the cyclic c-free product $\omega_1\prescript{\varphi_1}{\psi_1}{\circledast}^{\varphi_2}_{\psi_2}\omega_2 $ depends on the six linear functionals.
\begin{remark}
We do not assume positivity for the three linear functionals $\varphi, \psi, \omega$ (this is why also we do note assume $\mathcal{A}$ to be equipped with an involution). The authors in \cite{arizmendi2022cyclic} introduced the notion of \emph{cyclic non-commutative probability space} (cncps), in our case the appropriate notion would be the one of \emph{cyclic-conditional probability space}, that is the data of a $\star$-algebra $\mathcal{A}$ over the complex numbers and of a triple of positive linear functionals $(\psi,\varphi,\omega)$ with $\omega$ tracial. The distribution of a self-adjoint element of $\mathcal{A}$ is the set of numbers $\{\psi(a^n),\varphi(a^n),\omega(a^n)\}$.
\end{remark}
\begin{example} If $a$ and $b$ are two cyclically c-free random variables then
\begin{align*}
\omega(ab)&=\omega((a-\psi(a))(b-\psi(b))) + \psi(a)\omega(b) + \omega(a)\psi(b) - \omega({1}_{\mathcal{A}})\psi(a)\psi(b)\\
&=\varphi((a-\psi(a)))\varphi((b-\psi(b))) + \psi(a)\omega(b) + \omega(a)\psi(b) - \omega({1}_{\mathcal{A}})\psi(a)\psi(b)\\
&=\varphi(a)\varphi(b)-\psi(a)\varphi(b)-\varphi(a)\psi(b)+\psi(a)\psi(b) + \psi(a)\omega(b) + \omega(a)\psi(b) - \omega({1}_{\mathcal{A}})\psi(a)\psi(b).
\end{align*}
\end{example}

\newcommand{\cas}[4]{\,\prescript{#1}{#2}{\circledast}^{#3}_{#4}\,}
\newcommand{\vp}{\varphi}
\begin{proposition}
The cyclically c-free product is associative.
\end{proposition}
\begin{proof}
The proof is standard. Let $(\psi_i,\varphi_i, \omega_i)$ be three tiples of linear functionals with $\omega_1(1_{\mathcal{A}_1})=\omega_2(\mathcal{A}_2) = \omega_3(1_{\mathcal{A}_3})$.
We pick a cyclically alternated word $a_1\cdots a_n$ with $a=a_j\in\mathcal{A}_{i_j}$. We suppose that each $a_i$ is centered with respect to $\psi$ ($\psi_{i_j}$ if $a_i \in \mathcal{A}_{i_j}$). Then, by grouping together consecutive letters in the word $a$ belonging to $\mathcal{A}_1\star\mathcal{A}_2$, we may write $a=b_1\cdots b_p$ and each $b_i$ is centered respectively to $\psi_1*\psi_2$ (if $b_i\in\mathcal{A}_1\star\mathcal{A}_2$) or to $\psi_3$ (if $b_i \in \mathcal{A}_3$). Note that $(b_1,\ldots,b_n)$ may not be cyclically alternated; it is if (and only if) either $a_1$ or $a_n$ belongs to the algebra $\mathcal{A}_3$.

Let us suppose this is the case, and moreover that $a_n\in\mathcal{A}_3$.
We get 
\begin{align*}
    (\omega_1 \cas{\vp_1}{\psi_1}{\psi_2}{\vp_2} \omega_2) \cas{\vp_1{}_{\psi_1}*_{\psi_2}\vp_2}{\psi_1* \psi_2~~}{\varphi_3}{\psi_3} \omega_3(a_1\cdots a_n)
    &=\vp_1{}_{\psi_1}*_{\psi_2}\vp_2(b_1)\cdots \varphi_3(a_n) \\
    &=\varphi_{i_1}(a_1)\varphi_{i_2}(a_2)\cdots \varphi_3(a_n).
\end{align*}
On the contrary, if both $a_1$ and $a_n$ belongs to $\mathcal{A}_1\star\mathcal{A}_2$ and, say, $a_1\in\mathcal{A}_1$ and $a_n\in\mathcal{A}_2$
\begin{equation*}
    (\omega_1 \cas{\vp_1}{\psi_1}{\psi_2}{\vp_2} \omega_2) \cas{\vp_1{}_{\psi_1}\star_{\psi_2}\vp_2}{\psi_1* \psi_2~~}{\varphi_3}{\psi_3} \omega_3(a_1\cdots a_n) = \vp_1{}_{\psi_1}*_{\psi_2}\vp_2(b_pb_1)\cdots \varphi_3(b_{p-1}).
\end{equation*}
Since $b_pb_1$ is an alternated word on the letters $a's$, one gets the right-hand side of the last equality is equal to
\begin{equation*}
    \varphi_{i_{j_1}}(a_{j_1})\cdots \varphi_2(a_{j_k})\varphi_1(a_1)\cdots \varphi_{i_{l_p}}(a_{l_p}) \cdots \varphi_3(b_{p-1}) = \varphi_{i_1}(a_1)\cdots \varphi_{i_n}(a_n)
\end{equation*}
where $b_p = a_{j_1}\cdots a_{j_k}$ and $b_1=a_{l_1}\cdots a_{l_p}$. The same reasoning applies to compute 
\begin{equation*}
    \omega_1 \cas{\vp_1}{\psi_1}{\vp_2{}_{\psi_2}*_{\psi_3}\vp_3}{\psi_2*\psi_3} (\omega_2 \cas{\vp_2}{\psi_2}{\psi_3}{\vp_3} \omega_3).
\end{equation*}
\end{proof}
\subsection{Link with other cyclic independences}
Before explaining to which extent we recover already-known independences, we swiftly recall the connections between c-freeness and the Boolean, monotone and free independences. Let $\mathcal{A}$ be a unital complex algebra equipped with one unital linear functional $\varphi:\mathcal{A}\to \mathbb{C}$. By definition, unital sub-algebras  $(\mathcal{A}_i\subset \mathcal{A}:i\in I)$ are free with respect to $\varphi$ if and only if they are c-free with respect to $(\varphi,\varphi)$. In particular, we have $\varphi_1  \prescript{}{\varphi_1}{\ast}^{}_{\varphi_2}\varphi_2=\varphi_1*\varphi_2$ for c-free product of unital linear functionals. To relate the cyclically c-free product $\circledast$ to Boolean and monotone products, we 
restrict to augmented algebras.

Let $\mathcal{A}_1^0$ and $\mathcal{A}_2^0$ be two algebras and $\mathcal{A}_1$ and $\mathcal{A}_2$ their unitizations defined by $\mathcal{A}_i=\mathbb{C}1\oplus \mathcal{A}_i^0$. The linear functionals $\delta_i$ on $\mathcal{A}_i$ are defined by $\delta_i(\lambda\cdot 1+a^0)=\lambda$ for $\lambda\in \mathbb{C}$ and $a^0\in \mathcal{A}_i^0$. From now on, algebras are always assumed to be augmented when we use $\delta_i$.
In this setting, $ \mathcal{A}_1^0$ and  $\mathcal{A}_2^0$ are \emph{Boolean independent} in  $\mathcal{A}_1\star\mathcal{A}_2$ with respect to $\varphi_1  \prescript{}{\delta_1}{\ast}^{}_{\delta_2}\varphi_2$, \emph{monotone independent} with respect to $\varphi_1  \prescript{}{\delta_1}{\ast}^{}_{\varphi_2}\varphi_2$ and \emph{anti-monotone independent} with respect to $\varphi_1  \prescript{}{\varphi_1}{\ast}^{}_{\delta_2}\varphi_2$. We refer to \cite{hasebe2010new} for more details, and we emphasize the fact that the c-free product $\varphi_1  \prescript{}{\psi_1}{\ast}^{}_{\psi_2}\varphi_2$ reduces to well-known product of states in the cases where $\psi_i=\varphi_i$ or $\psi_i=\delta_i$.

Let us play the same game here, and consider the cyclic c-product in the cases where $\psi_i=\varphi_i$ or $\psi_i=\delta_i$. Let $\mathcal{A}_1=\mathbb{C}1\oplus \mathcal{A}_1^0$ and $\mathcal{A}_2=\mathbb{C}1\oplus \mathcal{A}_2^0$ two augmented unital algebras equipped with two triples of linear functionals $\psi_1,\varphi_1,\omega_1 :\mathcal{A}_1\to \mathbb{C}$ and $\psi_2,\varphi_2,\omega_2 :\mathcal{A}_2\to \mathbb{C}$ such that $\psi_1(1_{\mathcal{A}_1})=\psi_2(1_{\mathcal{A}_2})=1$, $\varphi_1(1_{\mathcal{A}_1})=\varphi_2(1_{\mathcal{A}_2})=1$, $\omega_1(1_{\mathcal{A}_1})=\omega_2(1_{\mathcal{A}_2})$,  and $\omega_1,\omega_2$ are tracial. For notional convenience, we set
$$(\psi,\varphi,\omega):= (\psi_1*\psi_2,\varphi_1\prescript{}{\psi_1}{*}^{}_{\psi_2}\varphi_2,\omega_1\prescript{\varphi_1}{\psi_1}{\circledast}^{\varphi_2}_{\psi_2}\omega_2 )$$ in such a way that $\mathcal{A}_1\subset\mathcal{A}_1\star\mathcal{A}_2$ and $\mathcal{A}_2 \subset \mathcal{A}_1\star\mathcal{A}_2$ are cyclically c-free with respect to $(\psi,\varphi,\omega)$.

The definition of c-freeness implies immediately the following.
\begin{proposition}
    With the notations introduced so far, we assume that $\psi_1=\delta_1$ and $\psi_2=\delta_2$. For any sequence $ (a_1,\ldots,a_n)$ of  $\mathcal{A}_1^0\cup \mathcal{A}_2^0$ which is alternating, we have
$$ \varphi(a_1\cdots a_n) = \varphi(a_1) \cdots \varphi(a_n).$$
For any sequence $ (a_1,\ldots,a_n)$ of  $\mathcal{A}_1^0\cup \mathcal{A}_2^0$ which is cyclically alternating, we have
$$ \omega(a_1\cdots a_n) = \varphi(a_1) \cdots \varphi(a_n).$$
In other words, $\mathcal{A}_1^0$ and $\mathcal{A}_2^0$ are \emph{cyclic-Boolean independent} in the sense of \cite[Definition 3.3]{arizmendi2022cyclic} with respect to $$(\varphi,\omega)=(\varphi_1\prescript{}{\delta_1}{*}^{}_{\delta_2}\varphi_2,\omega_1\prescript{\varphi_1}{\delta_1}{\circledast}^{\varphi_2}_{\delta_2}\omega_2).$$
\end{proposition}
Similarly, we obtain the cyclic-monotone independence as follows.
\begin{proposition}
    With the notations introduced so far, we assume that $\psi_1=\delta_1$ and $\psi_2=\varphi_2$. For any $a_1,\ldots,a_n\in \mathcal{A}_1^0$ and $b_0,\ldots,b_n\in \mathcal{A}_2$, we have
$$ \varphi(b_0a_1b_1\cdots a_nb_n) = \varphi(a_1\cdots a_n)\varphi(b_0)\cdots \varphi(b_n)$$
and
$$ \omega(b_0a_1b_1\cdots a_nb_n) = \omega(a_0\cdots a_n) \varphi(b_1)\cdots \varphi(b_{n-1})\varphi(b_nb_0).$$
In other words, $(\mathcal{A}_1^0,\mathcal{A}_2)$ are \emph{cyclic-monotone independent} in the sense of \cite[Definition 3.2]{collins2018free}, and in the sense of \cite[Definition 7.2]{arizmendi2022cyclic}, with respect to $$(\varphi,\omega)=(\varphi_1\prescript{}{\delta_1}{*}^{}_{\varphi_2}\varphi_2,\omega_1\prescript{\varphi_1}{\delta_1}{\circledast}^{\varphi_2}_{\varphi_2}\omega_2).$$
\end{proposition}
\begin{proof}The first relation
$$ \varphi(b_0a_1b_1\cdots a_nb_n) = \varphi(a_1\cdots a_n)\varphi(b_0)\cdots \varphi(b_n)$$
is due to the monotone independence of $(\mathcal{A}_1^0,\mathcal{A}_2)$ with respect to $\varphi=\varphi_1\prescript{}{\delta_1}{*}^{}_{\varphi_2}\varphi_2$, which has been proved by Franz in~\cite{Franz2005}.

Let us prove the second relation. Following \cite{collins2018free}, there exists a tracial linear functional $\omega_1 \unrhd \varphi_2:I\to \mathbb{C}$ (called the cyclic-monotone product) such that the wanted relation holds: for any $a_1,\ldots,a_n\in \mathcal{A}_1^0$ and $b_0,\ldots,b_n\in \mathcal{A}_2$, we have
\begin{equation}
    \omega_1 \unrhd \varphi_2(b_0a_1b_1\cdots a_nb_n) = \omega(a_1\cdots a_n)\varphi(b_1)\cdots \varphi(b_{n-1})\varphi(b_{n}b_0).\label{cyclicmonotone}
\end{equation}
We extend $\omega_1 \unrhd \varphi_2$ to $\mathcal{A}_1\star\mathcal{A}_2$ by $\omega_1 \unrhd \varphi_2(b)=\omega_2(b)$ if $b\in \mathcal{A}_2$.
This implies that, for any sequence $(a_1,\ldots,a_n)$ in $\mathcal{A}_1\cup \mathcal{A}_2$ which is cyclically alternating and centred with respect to $\psi$, we have
$$ \omega_1 \unrhd \varphi_2(a_1\cdots a_n) = 0=[\omega_1 \unrhd \varphi_2(a_1)]\cdots [\omega_1 \unrhd \varphi_2(a_n)].$$
Therefore, $\mathcal{A}_1$ and $\mathcal{A}_2$ are also c-free with respect to $(\psi,\varphi,\omega_1 \unrhd \varphi_2)$ which means that $\omega_1 \unrhd \varphi_2=\omega$. Replacing $\omega_1 \unrhd \varphi_2$ by $\omega$ in Equation~\eqref{cyclicmonotone} yields the result.
\end{proof}
It remains the case where $\psi_1=\varphi_1$ and $\psi_2=\varphi_2$. We recall that the c-free product coincides with the free product: $\varphi_1  \prescript{}{\varphi_1}{\ast}^{}_{\varphi_2}\varphi_2=\varphi_1*\varphi_2$. Due to this perspective, and following the two last propositions, it is tempting to say that $\mathcal{A}_1$ and $\mathcal{A}_2$ are \emph{cyclic free} with respect to $(\varphi_1*\varphi_2,\omega_1\prescript{\varphi_1}{\varphi_1}{\circledast}^{\varphi_2}_{\varphi_2}\omega_2)$. Quite surprisingly, in the case of tracial linear functional, we obtain in fact infinitesimal freeness, as defined by Belinschi and Shlyakhtenko in~\cite{belinschi2012free}.

Because infinitesimal freeness with respect to $(\varphi,\varphi')$ is in general defined assuming that $\varphi'(1)=0$, let us give a slighty more general definition here.
\begin{definition}\label{def:inffree}Let $\mathcal{A}$ be a unital complex algebra equipped with two linear functionals $\varphi,\varphi':\mathcal{A}\to \mathbb{C}$ where $\varphi$ is unital.
Let $\mathcal{A}_1 ,\mathcal{A}_2 \subset \mathcal{A}$ be two unital subalgebras. The two subalgebras $\mathcal{A}_1$ and $\mathcal{A}_2$ are \emph{infinitesimally free} with respect to $(\varphi,\varphi')$ if, for any sequence $ (a_1,\ldots,a_n)$ of  $\mathcal{A}_1\cup \mathcal{A}_2$ which is alternating and centered with respect to $\varphi$, we have
\begin{align*}
    \varphi(a_1\cdots a_n) &= 0,\\
    \text{and}\ \ \ \  \varphi^{\prime}(a_1\cdots a_n) &= \sum_{i=1}^n \varphi^{\prime}(a_i)\varphi(a_1\cdots a_{i-1} \cdot a_{i+1}\cdots a_n)
.
\end{align*}
\end{definition}
\begin{proposition}
    With the notations introduced so far, we assume that $\psi_1=\varphi_1$ and $\psi_2=\varphi_2$. For any sequence $ (a_1,\ldots,a_n)$ of  $\mathcal{A}_1\cup \mathcal{A}_2$ which is alternating and centered with respect to $\varphi$, we have
\begin{align*} \varphi(a_1\cdots a_n) &= 0,\\
    \text{and}\ \ \ \  \omega(a_1\cdots a_n) &= \sum_{i=1}^n\omega(a_i)\varphi(a_{i+1}\cdots  a_n\cdot a_1\cdots a_{i-1}).\end{align*}
In particular, whenever $\varphi$ is tracial, $\mathcal{A}_1$ and $\mathcal{A}_2$ are \emph{infinitesimally free} in the sense of Definition~\ref{def:inffree} with respect to $$(\varphi,\omega)=(\varphi_1\prescript{}{\varphi_1}{*}^{}_{\varphi_2}\varphi_2,\omega_1\prescript{\varphi_1}{\varphi_1}{\circledast}^{\varphi_2}_{\varphi_2}\omega_2).$$
\end{proposition}

\begin{proof}The first equation is immediate, and shows the freeness of $\mathcal{A}_1$ from $\mathcal{A}_2$ with respect to $\varphi$. In particular, using \cite[Lemma 5.18]{nica2006lectures}, we know that for sequences $ (a_1,\ldots,a_n)$ and $(b_1,\ldots,b_m)$ of  $\mathcal{A}_1\cup \mathcal{A}_2$ which are alternating and centered with respect to $\varphi$, we have
\begin{equation*}
    \varphi(a_1\ldots a_n b_1\ldots b_m)=\varphi(a_n b_1)\cdots \varphi(a_1 b_m)\label{twowords}
\end{equation*}
if $m=n$ and $\varphi(a_1\ldots a_n b_1\ldots b_m)=0$ if $m\neq n$.

For the second equality, we distinguish two cases. If $n$ is even, $(a_1,\cdots,a_n)$ is cyclically alternated, so
$$\omega(a_1\cdots a_n) = 0 = \sum_{i=1}^n\omega(a_i)\varphi(a_{i+1}\cdots  a_n\cdot a_1\cdots a_{i-1}).$$
If $n$ is even, we prove the result by induction. The case $n=1$ is immediate. For $n$ even and greater than $2$, we can write
\begin{align*}
    \omega(a_1\cdots a_n) =& \omega(a_na_1\cdots a_{n-1})\\
    =&\omega((a_na_1-\varphi(a_na_1))a_2\cdots a_{n-1})+\varphi(a_na_1)\omega(a_2\cdots a_{n-1})\\
    =&0+\varphi(a_na_1)\sum_{i=2}^{n-1}\omega(a_i)\varphi(a_{i+1}\cdots  a_{n-1}\cdot a_2\cdots a_{i-1})\\
    =&\varphi(a_na_1)\omega(a_{(n+1)/2})\varphi(a_{n-1}a_2)\cdots \varphi(a_{(n+3)/2}a_{(n-1)/2})\\
    =&\omega(a_{(n+1)/2})\varphi(a_{(n+3)/2}\cdots  a_{n}\cdot a_1\cdots a_{(n-1)/2})\\
    =&\sum_{i=1}^{n}\omega(a_i)\varphi(a_{i+1}\cdots  a_{n}\cdot a_1\cdots a_{i-1}).
\end{align*}
\end{proof}
In case of traciality, the infinitesimal freeness reduces in fact to a simpler condition as we show below (see also \cite[Lemma 2.2]{shlyakhtenko2018free} which is essentially the same).
\begin{lemma}Let $\mathcal{A}$ be a unital complex algebra equipped with two linear functionals $\varphi,\varphi':\mathcal{A}\to \mathbb{C}$ where $\varphi$ is unital.
Let $\mathcal{A}_1 ,\mathcal{A}_2 \subset \mathcal{A}$ be two unital subalgebras. Whenever $\varphi$ and $\varphi'$ are tracial, the following statements are equivalent:
\begin{enumerate}
    \item[(i)] $\mathcal{A}_1$ and $\mathcal{A}_2$ are infinitesimally free with respect to $(\varphi,\varphi')$;
    \item[(ii)] $\mathcal{A}_1$ and $\mathcal{A}_2$ are cyclically c-free with respect to $(\varphi,\varphi,\varphi')$;
    \item[(iii)] For any sequence $ (a_1,\ldots,a_n)$ of  $\mathcal{A}_1\cup \mathcal{A}_2$ which is cyclically alternating and centered with respect to $\varphi$, we have
$$ \varphi(a_1\cdots a_n) = \varphi'(a_1\cdots a_n)= 0.$$
\end{enumerate}
\end{lemma}
\begin{proof}We first note that in each statement, $\mathcal{A}_1$ and $\mathcal{A}_2$ are free with respect to $\varphi$. The definition of cyclic c-freeness yields the equivalence between the two last statements, as
$\varphi'(a_1\cdots a_n)=0=\varphi(a_1)\cdots\varphi(a_n)$.

Let us prove that $(i)$ and $(iii)$ are equivalent.
If $\mathcal{A}_1$ and $\mathcal{A}_2$ are infinitesimally free then for any cyclically alternating product of centered elements, taking off one variable $a_i$ results in an alternated word (up to cyclicity), thus
$$
\varphi^{\prime}(a_1\cdots a_n) =  \sum_{i=1}^n \varphi^{\prime}(a_i)\varphi(a_1\cdots a_{i-1} \cdot a_{i+1}\cdots a_n)=\sum_{i=1}^n \varphi^{\prime}(a_i)\varphi(a_{i+1}\cdots a_n\cdot a_1\cdots a_{i-1} )=0.
$$
Conversely, if $(iii)$ holds, for any cyclically alternating product of centered elements,
$$
\varphi^{\prime}(a_1\cdots a_n) =0=  \sum_{i=1}^n \varphi^{\prime}(a_i)\varphi(a_1\cdots a_{i-1} \cdot a_{i+1}\cdots a_n).$$
It remains the case of an alternating product of centered elements, which is not cyclically alternating. Pick $a_1\cdots a_n$ an alternating product of centered elements such that $a_1, a_n$ belongs to the same sub-algebra. Then
$
(a_n\cdot a_1 - \varphi(a_1a_n))a_2\cdots a_{n-1}
$
is cyclically alternating, thus $\varphi^{\prime}((a_n\cdot a_1 - \varphi(a_na_1))a_2\cdots a_n)=0$ and by a direct induction
\begin{align*}
    \varphi^{\prime}(a_1a_2\cdots a_n) &=0+\varphi(a_1a_n)  \varphi^{\prime}(a_2\cdots a_{n-1})\\
    &=\cdots \\
    &= \varphi(a_1a_n)\varphi(a_2a_{n-1})\cdots \varphi(a_{\frac{n-1}{2}}a_{\frac{n+2}{2}})\varphi^{\prime}(a_{\frac{n+1}{2}})\\
    &= \sum_{i=1}^n \varphi^{\prime}(a_i)\varphi(a_1\cdots a_{i-1} \cdot a_{i+1}\cdots a_n),
\end{align*}
where we used \eqref{twowords} for the last equality.
\end{proof}

\section{Asymptotic independences of random matrices}
\label{sec:asymptotic}
We come in this section to the core of our work; that is showing that the \emph{Vortex model} displays asymptotic conditional freeness and cyclic-conditional freeness.
For a family ${\bf A}_N = \{A_N^k: k \in K\}$ of $N\times N$ matrices, we recall that the non-commutative distribution of ${\bf A}_N$ is the linear functional $\psi_{{\bf A}_N}:\mathbb{C}\langle X_k:k\in K\rangle \to \mathbb{C}$ given by
$$\psi_{{\bf A}_N}(P)=\trn(P({\bf A}_N)).$$
The non-commutative distribution of ${\bf A}_N$ with respect to the vector state $v_N\in \mathbb{C}^N$ is the linear functional $\varphi_{{\bf A}_N}^{v_N}:\mathbb{C}\langle X_k:k\in K\rangle \to \mathbb{C}$ defined by
$$\varphi^{v_N}_{{\bf A}_N}(P)=\langle P({\bf A}_N) v_N,v_N\rangle=\Trn(v_Nv_N^{\star}P({\bf A}_N)).$$

\subsection{Asymptotic conditional freeness}

\begin{theorem}
\label{th:cfreeness}Let $R>0$, let $(v_N)_{N\geq 1}$ be a sequence of unit vectors of $\mathbb{C}^N$ and $U_{N} \in U(N)\cap \mathrm{Stab}(v_N)$ be a uniform random unitary matrix leaving $v_N$ invariant.

As $N$ tends to infinity, for any polynomial $P\in \mathbb{C}\langle X_k,Y_k:~k\in K\rangle$,
$$\mathbb{E}[\psi_{{\bf A}_N,U_N{\bf B}_NU_N^{\star}}(P)]=\psi_{{\bf A}_N}*\psi_{{\bf B}_N}(P)+O(N^{-1})$$
and
$$\mathbb{E}[\varphi^{v_N}_{{\bf A}_N,U_N{\bf B}_NU_N^{\star}}(P)]=\varphi^{v_N}_{{\bf A}_N}\prescript{}{\psi_{{\bf A}_N}\!\!}{*}^{}_{\psi_{{\bf B}_N}}\varphi^{v_N}_{{\bf B}_N}(P)+O(N^{-1})$$
    uniformly for the choice of any sequences ${\bf A}_N = \{A_N^k: k \in K\}$ and ${\bf B}_N = \{B^k_N:~k \in K\}$ of $N\times N$ matrices ($N\geq 1$) bounded in operator norm by $R$.
\end{theorem}
\begin{remark}\begin{itemize}
    \item 
We have chosen here to rotate randomly the family $\bfB_N$, but we could have done the same on $\bfA_N$ with the same results:
\begin{align*}
    &\mathbb{E}[\psi_{V_N{\bf A}_NV_N^{\star},U_N{\bf B}_NU_N^{\star}}(P)]=\psi_{{\bf A}_N}*\psi_{{\bf B}_N}(P)+O(N^{-1}), \\
    &\mathbb{E}[\varphi^{v_N}_{V_N{\bf A}_NV_N^{\star},U_N{\bf B}_NU_N^{\star}}(P)]=\varphi^{v_N}_{{\bf A}_N}\prescript{}{\psi_{{\bf A}_N}\!\!}{*}^{}_{\psi_{{\bf B}_N}}\varphi^{v_N}_{{\bf B}_N}(P)+O(N^{-1})
\end{align*}
where $V_N$ and $U_N$ are two independent uniform unitary matrices drawn from $U(N)\cap\mathrm{Stab}(v_N)$.
\item Whenever $v_N=(1/\sqrt{N},\ldots,1/\sqrt{N})$, the group $U(N)\cap \mathrm{Stab}(v_N)$ is in fact the bistochastic group. The theorem states that conjugating by the bistochastic group does not only imply asymptotic freeness with respect to the normalized trace $\trn$ (as already shown by Gabriel in~\cite{gabriel2015combinatorial}) but also asymptotic $c$-freeness with respect to the pair trace/antitrace $(\trn,\varphi^{v_N}(M))$, where the \emph{antitrace} is the unital linear functional $\varphi^{v_N}(M)=\frac{1}{N}\sum_{i,j}M_{ij}$ as in \cite{cebron2016traffic}.
\end{itemize}
\end{remark}

\begin{proof}For notational simplicity, we denote by $\psi_N$ the free product $\psi_{{\bf A}_N}*\psi_{{\bf B}_N}$ and by $\varphi_N$ the conditionally free product $\varphi^{v_N}_{{\bf A}_N}\prescript{}{\psi_{{\bf A}_N}\hspace{-0.2cm}}{*}^{}_{\psi_{{\bf B}_N}}\varphi^{v_N}_{{\bf B}_N}$.
This will ease the proof to allow the polynomial $P$ of the proposition to vary with $N\geq 0$. In fact, we prove a slightly more general result:
all estimates about the distributions of $\mathbf{A}_{N}$ and $\mathbf{B}_{N}$ (inequalities involving $\varphi_{\mathbf{A}_N}$ and $\psi_{\mathbf{B}_N}$) we will prove are true for sequences of polynomials $P = (P_N)_{N\geq 0}$ with coefficients and degrees bounded uniformly in $N$.

This makes possible to write any polynomial in $\mathbb{C}\langle X_k,Y_k:~k\in K \rangle$ as a sum of polynomials in $\mathbb{C}\langle X_k:~k\in K \rangle$, polynomials in $\mathbb{C}\langle Y_k:~k\in K \rangle$ and polynomials which are alternating and centered with respect to $\psi_N$ (this decomposition depending on $N$).

We consider a sequence of polynomials $(P_N)_{N\geq 0}$ with bounded degree and coefficients in $\mathbb{C}\langle X_k,Y_k:~k\in K \rangle$, such that each $P_N$ is alternating and centered with respect to $\psi_N$. For the convenience of the reader, we will drop here forth the dependence in $N$ and consider the case where the degree is constant and $P$ is starting by an element of $\mathbb{C}\langle X_k:~k\in K \rangle$:
$$
P=P_1(X)Q_1(Y)\cdots P_s(X)Q_s(Y)
$$
for polynomials $P_1,\ldots,Q_s$ such that the norms of their
coefficients are bounded uniformly in $N$, such that $(P_1(X),Q_1(Y),\ldots,P_s(X))$ is centred with respect to $\psi_N$, and  such that $\trn(Q_s({\bf B}_N))=0$ (if $P$ is cyclically alternating) or $Q_s(Y)=1$ (if not). Without loss of generality, let us assume that $v_N$ is the first vector $\varepsilon_1$ of the canonical basis of $\mathbb{C}^N$. Then,
$$U_N=\begin{pmatrix}
1& 0 \\
0& V_{N-1}\end{pmatrix},$$
where $V_{N-1}$ is uniformly distributed on $U(N-1)$. In particular, denoting by $p$ the $N\times (N-1)$ matrix
$$p= \begin{pmatrix}
0&0&\cdots&0\\
1 & 0 & \cdots & 0 \\
0 & \ddots & \ddots & \vdots \\
\vdots & \ddots & \ddots & 0 \\
0 & \cdots & 0 & 1 \end{pmatrix},$$
we have
$U_N= \varepsilon_1\varepsilon_1^{\star}+pV_{N-1}p^{\star}$.
Note that
$$P({\bf A}_N, U_N{\bf B}_NU_N^{\star}) = P_1({\bf A}_N)U_NQ_1({\bf B}_N)U_N^{\star}\cdots P_s({\bf A}_N)U_NQ_s({\bf B}_N)U_N^\star.$$
We record the positions (in the polynomial above) of each $U_N$ and $U_N^\star$ with the sequence $\{1,\dots,s\}$. If $S=\{s_1 < \ldots < s_\ell\}$ is a subset of $\{1,2,\dots,s\}$, we denote by $R_S$ the polynomial
$$P_1({\bf A}_N)U_NQ_1({\bf B}_N)U_N^{\star}\cdots P_s({\bf A}_N)U_NQ_s({\bf B}_N)U_N^\star$$
where each $U_N$ (resp. $U_N^\star$) is replaced by by $pV_{N-1}p^\star$ (resp. $p V_{N-1}^\star p^\star$)  if his position is in $S$, and by $\varepsilon_1\varepsilon_1^{\star}$ if not.
Developing each $U_N$ in $\varepsilon_1\varepsilon_1^{\star}+pV_Np^{\star}$, we get
\begin{align*}P({\bf A}_N, U_N{\bf B}_NU_N^{\star})=&\sum_{S\subset \{1,\dots,s\}} R_S\\
=&P_1({\bf A}_N)\varepsilon_1\varepsilon_1^{\star}Q_1({\bf B}_N)\varepsilon_1\varepsilon_1^{\star}\cdots P_s({\bf A}_N)\varepsilon_1\varepsilon_1^{\star}Q_s({\bf B}_N)\varepsilon_1\varepsilon_1^{\star}+\sum_{\emptyset \neq S\subset \{1,\dots,s\}}R_S\\
=&P_1({\bf A}_N)\varepsilon_1\varepsilon_1^{\star}\cdot\varphi_N(Q_1)\cdots \varphi_N(P_s)\varphi_N(Q_s)+\sum_{\emptyset \neq S\subset \{1,\dots,s\}}R_S.
\end{align*}
%
In particular,
\begin{equation}\trn(P({\bf A}_N, U_N{\bf B}_NU_N^{\star}))=\frac{1}{N}\varphi_N(P_1)\varphi_N(Q_1)\cdots \varphi_N(P_s)\varphi_N(Q_s)+ \sum_{\emptyset \neq S\subset \{1,\dots,s\}}\trn(R_S)\label{eq:lesRS}\end{equation}
and
$$\Trn(\varepsilon_1\varepsilon_1^{\star}P({\bf A}_N, U_N{\bf B}_NU_N^{\star}))=\varphi_N(P_1)\varphi_N(Q_1)\cdots \varphi_N(P_s)\varphi_N(Q_s)+ \sum_{\emptyset \neq S\subset \{1,\dots,s\}}\Trn(\varepsilon_1\varepsilon_1^{\star}R_S).$$
The rest of the proof consists in showing that the terms in the sums tends to $0$ in expectation as $N$ tends to infinity.
Note that, for $S=\{s_1<\ldots<s_\ell\}\neq \emptyset$,
$R_S$ can be written as $$M_N(1)pV_{N-1}^{n_1}p^\star M_N(2)pV_{N-1}^{n_2}p^\star \cdots M_N(\ell)pV_{N-1}^{n_\ell}p^\star M_N(\ell+1)$$
with $n_1\ldots,n_\ell\in \{1,-1\}$. Each $M_N(k)$ either belongs to $$\{P_1({\bf A}_N),Q_1({\bf B}_N),\ldots, P_s({\bf A}_N),Q_s({\bf B}_N)\},$$ either is of rank one because of occurrences of $\varepsilon_1\varepsilon_1^{\star}$ (except possibly $M_N(\ell+1)=I_N$ if $s_\ell=s$).

On one hand, we have
\begin{align*}
    \trn(R_S)&=\frac{1}{N}\Trn\Big(M_N(1)pV_N^{n_1}p^\star M_N(2)pV_N^{n_2}p^\star \cdots M_N(\ell)pV_N^{n_\ell}p^\star M_N(\ell+1)\Big)\\
    &=\frac{1}{N}\Trn\Big( M_N(\ell+1)M_N(1)pV_N^{n_1}p^\star M_N(2)pV_N^{n_2}p^\star \ldots M_N(\ell)pV_N^{n_\ell}p^\star\Big)\\
    &=\frac{1}{N}\Trnm\Big(p^\star M_N(\ell+1)M_N(1)p \cdot V_N^{n_1}\cdot p^\star M_N(2)p\cdot V_N^{n_2} \ldots p^\star M_N(\ell)p\cdot V_N^{n_\ell}\Big).
\end{align*}
Note that the matrices $p^\star M_N(\ell+1)M_N(1)p$, $p^\star M_N(2)p,\ldots,p^\star M_N(\ell)p$ either belong to $$\{ p^\star P_1({\bf A}_N)p,p^\star Q_1({\bf B}_N)p,\ldots, p^\star P_s({\bf A}_N)p,p^\star Q_s({\bf B}_N)p\}$$
 either are of rank one. If $P$ is cyclically alternating, then $P_1({\bf A}_N),Q_1({\bf B}_N),\ldots, P_s({\bf A}_N),Q_s({\bf B}_N)$ have vanishing traces, so the normalized trace of  $p^\star P_1({\bf A}_N)p$, $p^\star Q_1({\bf B}_N)p$,\ldots, $p^\star P_s({\bf A}_N)p$, $p^\star Q_s({\bf B}_N)p$ is $O((N-1)^{-1})$. Indeed, we have for example
 $$\frac{1}{N-1}\Trnm(p^\star P_1({\bf A}_N)p)=-\frac{1}{N-1}\varphi^{v_N}(P_1({\bf A}_N))=O((N-1)^{-1}).$$
As a consequence, Proposition~\ref{prop:freenesswithunitaries} yields \begin{equation}\mathbb{E}[\trn(R_S)]=O((N-1)^{-2})=O(N^{-2}).\label{RStozero}\end{equation} Finally, for any cyclically alternating $P$ and centred with respect to $\psi_N$, and such that the norm of their
coefficients is bounded uniformly in $N$, we have
\begin{equation}\mathbb{E}\left[\trn(P({\bf A}_N, U_N{\bf B}_NU_N^{\star}))\right]=\frac{1}{N}\varphi_N(P_1)\varphi_N(Q_1)\cdots \varphi_N(P_s)\varphi_N(Q_s)+ O(N^{-2}).\label{eq:trace}\end{equation}
Note that we only proved
$$\mathbb{E}[\psi_{{\bf A}_N,U_N{\bf B}_NU_N^{\star}}(P)]=\psi_{{\bf A}_N}*\psi_{{\bf B}_N}(P)+O(N^{-1})$$
for cyclically alternating polynomials. However, the traciality of $\psi_{{\bf A}_N,U_N{\bf B}_NU_N^{\star}}$ and $\psi_{{\bf A}_N}*\psi_{{\bf B}_N}$, and the uniformity of our estimates in the degree and the coefficients, allow to conclude that
$$\mathbb{E}[\psi_{{\bf A}_N,U_N{\bf B}_NU_N^{\star}}(P)]=\psi_{{\bf A}_N}*\psi_{{\bf B}_N}(P)+O(N^{-1})$$
is true for any polynomial $P\in \mathbb{C}\langle X_k,Y_k:~k\in K \rangle$.

On the other hand, we have
\begin{align*}
    \Trn(\varepsilon_1\varepsilon_1^{\star} R_S)&=\Trn\Big(\varepsilon_1\varepsilon_1^{\star} M_N(1)pV_N^{n_1}p^\star M_N(2)pV_N^{n_2}p^\star \cdots M_N(\ell)pV_N^{n_\ell}p^\star M_N(\ell+1)\Big)\\
    &=\Trn\Big( M_N(\ell+1)\varepsilon_1\varepsilon_1^{\star} M_N(1)pV_N^{n_1}p^\star M_N(2)pV_N^{n_2}p^\star \ldots M_N(\ell)pV_N^{n_\ell}p^\star\Big)\\
    &=\Trnm\Big(p^\star M_N(\ell+1)\varepsilon_1\varepsilon_1^{\star} M_N(1)p \cdot V_N^{n_1}\cdot p^\star M_N(2)p\cdot V_N^{n_2} \ldots p^\star M_N(\ell)p\cdot V_N^{n_\ell}\Big).
\end{align*}
Note that the matrices $p^\star M_N(\ell+1)\varepsilon_1\varepsilon_1^{\star} M_N(1)p$, $p^\star M_N(2)p,\ldots,p^\star M_N(\ell)p$ either belong to $$\{ p^\star P_1({\bf A}_N)p,p^\star Q_1({\bf B}_N)p,\ldots, p^\star P_s({\bf A}_N)p,p^\star Q_s({\bf B}_N)p\}$$
 either are of rank one. If $P$ is cyclically alternating, then $P_1({\bf A}_N),Q_1({\bf B}_N),\ldots, P_s({\bf A}_N),Q_s({\bf B}_N)$ have vanishing traces and we can apply Proposition~\ref{prop:freenesswithunitaries} as before. We get $\mathbb{E}[\Trn(\varepsilon_1\varepsilon_1^{\star} R_S)]=O((N-1)^{-1})=O(N^{-1})$. If $Q_s=1$, we can still apply Proposition~\ref{prop:freenesswithunitaries} whenever $M_N(\ell)\neq Q_s$ and we get $\mathbb{E}[\Trn(\varepsilon_1\varepsilon_1^{\star} R_S)]=O((N-1)^{-1})=O(N^{-1})$. In the particular case where $M_N(\ell)= Q_s=1$ has a non-vanishing trace, we have $ M_N(\ell+1)=I_N$ and
$$
   \mathbb{E}\left[ \Trn(\varepsilon_1\varepsilon_1^{\star} R_S)\right]=\Trnm\Big(p^\star \varepsilon_1\varepsilon_1^{\star} M_N(1)p \cdot V_N^{n_1}\cdot p^\star M_N(2)p\cdot V_N^{n_2} \ldots p^\star M_N(\ell)p\cdot V_N^{n_\ell}\Big)=0
$$
because $p^\star \varepsilon_1=0$. Finally, we always get
$$\Trn(\varepsilon_1\varepsilon_1^{\star}P({\bf A}_N, U_N{\bf B}_NU_N^{\star}))=\varphi_N(P_1)\varphi_N(Q_1)\cdots \varphi_N(P_s)\varphi_N(Q_s)+O(N^{-1}).$$

We just proved that, for any polynomial $P\in \mathbb{C}\langle X_k,Y_k,~k\in K\rangle$,
$$\mathbb{E}[\psi_{{\bf A}_N,U_N{\bf B}_NU_N^{\star}}(P)]=\psi_{{\bf A}_N}*\psi_{{\bf B}_N}(P)+O(N^{-1})$$
and
$$\mathbb{E}[\varphi^{v_N}_{{\bf A}_N,U_N{\bf B}_NU_N^{\star}}(P)]=\varphi^{v_N}_{{\bf A}_N}\prescript{}{\psi_{{\bf A}_N}\!\!}{*}^{}_{\psi_{{\bf B}_N}}\varphi^{v_N}_{{\bf B}_N}(P)+O(N^{-1})$$
with uniform estimates on the operator norm of the matrices.
\end{proof}
We now proceed with almost sure estimates.
\begin{theorem}
\label{th:cfreeness_almost_sure}
Let ${\bf A}_N = \{A_N^k: k \in K\}$ and ${\bf B}_N = \{B^k_N:~k \in K\}$ be sequences of families of $N\times N$ matrices ($N\geq 2$) bounded in operator norm uniformly in $N$. Let $(v_N)_{N\geq 2}$ be a sequence of unit vectors of $\mathbb{C}^N$ and $U_{N} \in U(N)\cap \mathrm{Stab}(v_N)$ be a sequence ($N\geq 2$) of uniform random unitary matrix leaving $v_N$ invariant.

As $N$ tends to infinity, for any polynomial $P\in \mathbb{C}\langle X_k,Y_k:~k\in K\rangle$,
$$\psi_{{\bf A}_N,U_N{\bf B}_NU_N^{\star}}(P)=\psi_{{\bf A}_N}*\psi_{{\bf B}_N}(P)+o(1)\ \ \text{almost surely}$$
and
$$\varphi^{v_N}_{{\bf A}_N,U_N{\bf B}_NU_N^{\star}}(P)=\varphi^{v_N}_{{\bf A}_N}\prescript{}{\psi_{{\bf A}_N}\!\!}{*}^{}_{\psi_{{\bf B}_N}}\varphi^{v_N}_{{\bf B}_N}(P)+o(1)\ \ \text{almost surely.}$$
In other words, almost surely, the ensemble ${\bf A}_N$ and $U_N{\bf B}_NU_N^{\star}$ are asymptotically c-free with respect to $(\trn,\varphi^{v_N})$.
\end{theorem}
\begin{proof}Thanks to Theorem~\ref{th:cfreeness}, we are left to prove that $$\psi_{{\bf A}_N,U_N{\bf B}_NU_N^{\star}}(P)=\mathbb{E}[\psi_{{\bf A}_N,U_N{\bf B}_NU_N^{\star}}(P)]+o(1)\ \ \text{almost surely}$$
and
$$\varphi^{v_N}_{{\bf A}_N,U_N{\bf B}_NU_N^{\star}}(P)=\mathbb{E}[\varphi^{v_N}_{{\bf A}_N,U_N{\bf B}_NU_N^{\star}}(P)]+o(1)\ \ \text{almost surely,}$$
which is a direct consequence of the concentration phenomenon of Proposition~\ref{Lipschitz}.
\end{proof}
The following concentration phenomenon will be used several times, which explains its very general formulation.
\begin{proposition}
\label{Lipschitz}
Let ${\bf M}_N = \{M_N^k: k \in K\}$ be a sequence of families of $N\times N$ matrices ($N\geq 2$) bounded in operator norm uniformly in $N$. Let $(u_N)_{N\geq 2}$ and $(v_N)_{N\geq 2}$ be two sequences of unit vectors of $\mathbb{C}^N$ and $U_{N} \in U(N)\cap \mathrm{Stab}(u_N)$ be a sequence ($N\geq 2$) of uniform random unitary matrix leaving $u_N$ invariant.

Then, for any polynomial $P\in \langle X,X^{-1},Y_k;k\in K\rangle$, we have
$$\psi_{U_N,{\bf M}_N}(P)=\mathbb{E}[\psi_{U_N,{\bf M}_N}(P)]+o(1)\ \ \text{almost surely}$$
and
$$\varphi^{v_N}_{U_N,{\bf M}_N}(P)=\mathbb{E}[\varphi^{v_N}_{U_N,{\bf M}_N}(P)]+o(1)\ \ \text{almost surely.}$$

\end{proposition}

\begin{proof}
Without loss of generality, let us assume that $u_N$ is the first vector of the canonical basis of $\mathbb{C}^N$. Then, we have
$$U_N=\begin{pmatrix}
1& 0 \\
0& V_{N-1}\end{pmatrix},$$
where $V_{N-1}$ is uniformly distributed on $U(N-1)$. We want to use the concentration inequality of Theorem~\ref{th:concentration_unitary} for $V_{N-1}$. As the map
$$V\mapsto \begin{pmatrix}
1& 0 \\
0& V\end{pmatrix} $$ is an isometry for the Hilbert-Schmidt norm, it remains to prove that there exists a constant $C>0$ (independent from $N$) such that
$U\mapsto \psi_{U,{\bf M}_N}(P) $ is $\frac{C}{\sqrt{N}}$-Lipschitz
and that
$U\mapsto \varphi^{v_N}_{U,{\bf M}_N}(P) $
is $N^{1/4}C$-Lipschitz for the Hilbert-Schmidt metric $$d(U,V)=\sqrt{\Tr((U-V)(U-V)^\star)}$$ on the unitary group $U(N)$. We follow here the proof of \cite[Theorem 3.5]{cebron2022freeness}. We set $$f(U):=\psi_{U,{\bf M}_N}(P)=\trn(P(U,{\bf M}_N)).$$In order to bound $f(U)-f(V)$, we rewrite $f(U)-f(V)$ as a sum of traces by using the swapping trick in order to replace each occurrence of $U$ by $V$ (swapping
one term at a time make appears alternatively $U-V$ or $U^*-V^*$).

The non-commutative H\H{o}lder inequality says that, for any $M_1,\ldots,M_k\in \mathcal{M}_N(\mathbb{C})$ and any integers $n_1,\ldots,n_k$ such that $\sum_i 1/(2n_i)=1$, we have
$$\Big|\trn(M_1\ldots M_k)\Big|\leq \sqrt[\leftroot{-3}\uproot{3}2n_1]{\trn\Big((M_1M_1^\star)^{n_1}\Big)}\cdots \sqrt[\leftroot{-3}\uproot{3}2n_k]{\trn\Big((M_kM_k^\star)^{n_k}\Big)} $$ (see for example \cite[Theorem 2.1.5]{da2018lecture}). Using the non-commutative H\H{o}lder inequality with exponent $n_i=1$ for the term $(U-V)$, and the fact that matrices from $\mathbf{M}_N$ are bounded in operator norm,
we conclude that there exists $C>0$ such that
$$|f(U)-f(V)|\leq C\sqrt{\trn((U-V)(U-V)^\star)}=\frac{C}{\sqrt{N}}\sqrt{\Tr((U-V)(U-V)^\star)}.$$
Similarly,  we set $$g(U):=\frac{1}{N}\varphi^{v_N}_{U,{\bf M}_N}(P)=\trn(v_Nv_N^*P(U,{\bf M}_N)).$$
We proceed similarly using the exponent $n_i=2$ for the matrix $v_Nv_N^\star$, and obtain
 that there exists $C>0$ such that
$$|g(U)-g(V)|\leq C\sqrt[\leftroot{-3}\uproot{3}4]{\trn(v_Nv_N^\star)}\sqrt{\trn((U-V)(U-V)^\star)}=\frac{C}{N^{3/4}}\sqrt{\Tr((U-V)(U-V)^\star)}.$$
Consequently, $U\mapsto \varphi^{v_N}_{{\bf A}_N,U{\bf B}_NU^{\star}}(P)$ is $N^{1/4}$C-Lipschitz.

As a consequence, Theorem~\ref{th:concentration_unitary} applies for $V_{N-1}$. We get
$$\mathbb{P}\left[\left|\vphantom{\varphi^{v_N}_{U_N,{\bf M}_N}}\mathrm{Re}\ \psi_{U_N{\bf M}_N}(P)-\mathbb{E}[\mathrm{Re}\ \psi_{U_N{\bf M}_N}(P)]\right|\geq \delta\right]\leq 2e^{-\frac{N(N-1)\delta^2}{12C^2}},$$
$$\mathbb{P}\left[\left|\mathrm{Re}\ \varphi^{v_N}_{U_N{\bf M}_N}(P)-\mathbb{E}[\mathrm{Re}\ \varphi^{v_N}_{U_N{\bf M}_N}(P)]\right|\geq \delta\right]\leq 2e^{-\frac{(N-1)\delta^2}{12C^2\sqrt{N}}}$$
and similarly for the imaginary parts. We deduce the wanted almost sure convergences, by Borel-Cantelli.\end{proof}

\subsection{Asymptotic cyclic-conditional freeness}
In this section, we prove that the infinitesimal distribution of the Vortex model obeys to cyclic-conditional freeness.

We consider as usual now two deterministic sequences ${\bf A}_N = \{A_N^k: k \in K\}$ and ${\bf B}_N = \{B^k_N:~k \in K\}$ of $N\times N$ matrices ($N\geq 1$)  uniformly bounded in operator-norm. We assume convergence of the distributions of $\bfA_N$ and $\bfB_N$ and existence of $\frac{1}{N}$-expansions. That is, for any
polynomial $P\in \mathbb{C}\langle X_k:~k\in K\rangle$,
\begin{equation}\label{eq:expansion}
\psi_{{\bf A}_N}(P) = \psi_{\bf A}(P) + \frac{1}{N}\omega_{\bf A}(P) + o(1/N), \ \ \psi_{{\bf B}_N}(P) = \psi_{\bf B}(P) + \frac{1}{N}\omega_{\bf B}(P) + o(1/N)
\end{equation}
where $\psi_{\bf A}$, $\omega_{\bf A}$, $\psi_{\bf B}$ and $\omega_{\bf B}$ are linear functionals on $\mathbb{C}\langle X_k:~k\in K\rangle$.
Notice that since $\trn$ is tracial, $\omega_{\bfA}$ and  $\omega_{\bf B}$ are also tracial and that $\omega_{\bfA}(1)=\omega_{\bfB}(1)=0$.

We suppose existence of a deterministic sequence of unitary vectors $(v_N)_{N\geq 1}$ such that the associated states $\varphi^{v_N}_{\bfA_N}$ and $\varphi^{v_N}_{\bfB_N}$ converge. That is, for any
polynomial $P\in \mathbb{C}\langle X_k:~k\in K\rangle$, \begin{equation}\label{eq:expansiondeux}
\varphi^{v_N}_{{\bf A}_N}(P) = \varphi_{\bf A}(P) +o(1), \ \ \varphi^{v_N}_{{\bf B}_N}(P) = \varphi_{{\bf B}}(P) + o(1).
\end{equation} where $\varphi_{\bf A}$ and $\varphi_{\bf B}$ are linear functionals on $\mathbb{C}\langle X_k:~k\in K\rangle$.

Then, in addition to the result of the previous section the following proposition holds.

\begin{theorem}\label{th:cycliccfreeness}
Let $(v_N)_{N\geq 2}$ be a sequence of unit vectors of $\mathbb{C}^N$. Let ${\bf A}_N = \{A_N^k: k \in K\}$ and ${\bf B}_N = \{B^k_N:~k \in K\}$ be two sequences of $N\times N$ matrices ($N\geq 2$) as described above, i.e.   uniformly bounded in operator-norm and satisfying \eqref{eq:expansion} and \eqref{eq:expansiondeux}. Let $U_{N} \in U(N)\cap \mathrm{Stab}(v_N)$ be a sequence ($N\geq 2$) of uniform random unitary matrix leaving $v_N$ invariant.

Then, for any
polynomial $P\in \mathbb{C}\langle X_k,Y_k:~k\in K\rangle$, we have $$
\mathbb{E}\left[\frac{1}{N}\mathrm{Tr}[P({\bf A}_N,U_N{\bf B}_NU_N^{\star})] \right]= \psi_{\bf A} * \psi_{\bf B}(P)  + \frac{1}{N} \omega_{\bfA} \prescript{\psi_{\bfA}\!\!}{\varphi_{\bfA}\!\!}{}\circledast^{\psi_{\bfB}}_{\varphi_{\bfB}} \omega_{\bfB}(P)+ O(N^{-2}).
$$
Equivalently, for any $(P_1,Q_1,\ldots,P_s,Q_s)$ sequence of polynomial cyclically alternating in $ \mathbb{C}\langle X_k: k\in K\rangle \cup \mathbb{C}\langle Y_k: k\in K\rangle$ and centered with respect to $\psi_{\bf A} * \psi_{\bf B}$, we have
\begin{equation}
	\label{eqn:toprove}
\mathbb{E}\left[	\Trn\Big(P_1(\bfA_N)Q_1(U_N\bfB_NU_N^{\star})\cdots P_s(\bfA_N)Q_s(U_N\bfB_NU_N^{\star})\Big) \right]=
    	 \varphi_{\bfA}(P_1)\varphi_{\bfB}(Q_1)\cdots \varphi_{\bfA}(P_s)\varphi_{\bfB}(Q_s)+ O(N^{-1}).
\end{equation}
\end{theorem}
\begin{remark} 
\begin{enumerate}
    \item Whenever the limits
    $$
    \lim_{N\to \infty} \mathrm{Tr}_N(P({\bf A}_N)), \quad \lim_{N\to \infty}\mathrm{Tr}_N(P({\bf B}_N))
    $$
    exist for any polynomial $P$, and thus $\lim_{N}\trn(P({\bf A}_N))=\lim_{N}\trn(P({\bf B}_N))=0$, one recovers cyclic-Boolean independence,
  which translates concretely to the following equality for cyclically alternating polynomials:
    \begin{equation*}
        \lim_{N\to \infty}\mathbb{E}\left[\Trn(P_1({\bf A}_N) Q_1(U_N{\bf B}_NU_N^\star)\cdots)\right]= \prod_i\lim_{N\to \infty}\langle v_N, P_i({\bf A}_N)v_N\rangle\prod_i\lim_{N\to \infty}\langle v_N, Q_i({\bf B}_N)v_N\rangle .
    \end{equation*}
    \item Whenever the limits
    \begin{equation*}
        \lim_{N\to \infty}\Trn(P({\bf A}_N)), \quad \lim_{N\to \infty}\trn(P({\bf B}_N))= \lim_{N\to \infty}\langle v_N, P({\bf B}_N) v_N\rangle
    \end{equation*}
    exist for any polynomial $P$, one recovers cyclic-monotone independence,  which translates concretely to the following equality for cyclically alternating polynomials:
    \begin{equation*}
    \lim_{N\to \infty}\mathbb{E}\left[\mathrm{Tr}_N(P_1({\bf A}_N)Q_1(U_N {\bf B}_NU_N^{\star}) \cdots) \right]=  \lim_{N\to \infty}\mathrm{Tr}_N(P_1({\bf A}_N)P_2({\bf A}_N)\cdots )\prod_{i} \lim_{N\to \infty}\trn(Q_i({\bf B}_N)).
    \end{equation*}
\end{enumerate}
\end{remark}

\begin{proof}(sketch of the proof) Let us prove the result for a product of polynomials cyclically alternating  and centered with respect to $\psi_{\bf A} * \psi_{\bf B}$.
 Consider the case  where $$
P=P_1(X)Q_1(Y)\cdots P_s(X)Q_s(Y)
$$
for polynomials $P_1,\ldots,P_s\in \mathbb{C}\langle X_k: k\in K\rangle$ and $Q_1,\ldots,Q_s\in \mathbb{C}\langle X_k: k\in K\rangle$ which are centered with respect to $\psi_{\bf A} * \psi_{\bf B}$. We unfold exactly the same lines of arguments as in the proof of Theorem \ref{th:cfreeness}, and obtain \eqref{eq:trace}, that is
\begin{equation*}\mathbb{E}\left[\trn(P({\bf A}_N, U_N{\bf B}_NU_N^{\star}))\right]=\frac{1}{N}\varphi^{v_N}_{{\bf A}_N}(P_1)\varphi^{v_N}_{{\bf B}_N}(Q_1)\cdots \varphi^{v_N}_{{\bf A}_N}(P_s)\varphi^{v_N}_{{\bf B}_N}(Q_s)+ O(N^{-2}).\end{equation*}
Because of the convergence of $\varphi^{v_N}_{{\bf A}_N}$ and $\varphi^{v_N}_{{\bf B}_N}$, we obtain the wanted convergence \eqref{eqn:toprove}.

The general result for $P\in \mathbb{C}\langle X_k,Y_k:~k\in K\rangle$ follows by decomposing (thanks to traciality) any trace of polynomials as a sum of traces of  products of polynomials which are cyclically alternating  and centered with respect to $\psi_{\bf A} * \psi_{\bf B}$.\end{proof}

We now proceed with almost sure estimates. If $U_N$ is a unitary in $U(N)\cap \mathrm{Stab}(v_n)$, we let $U_N^{\perp}$ be its compression on
the orthogonal complement of $\langle v_N\rangle$, that is $$U_N^{\perp}:=(I_N-v_Nv_N^{\star})U_N(I_N-v_Nv_N^{\star})=U_N-v_Nv_N^\star.$$
\begin{corollaire}
\label{cor:fluctuat}
Let $k\geq 2$ an integer. Let $(P_1,Q_1,\ldots,P_s,Q_s)$ be a cyclically alternated sequence of polynomials centered  with respect to $\psi_{\bfA}*\psi_{\bfB}$. Under the hypothesis and notations of Theorem~\ref{th:cycliccfreeness}, we have almost surely
\begin{multline*}
    	\Trn\Big(P_1(\bfA_N)Q_1(U_N\bfB_NU_N^{\star})\cdots P_s(\bfA_N)Q_s(U_N\bfB_NU_N^{\star})\Big) =
    	 \varphi_{\bfA}(P_1)\varphi_{\bfB}(Q_1)\cdots \varphi_{\bfA}(P_s)\varphi_{\bfB}(Q_s)
    \\	+ \Trn\Big(P_1(\bfA_N)U_N^{\perp}Q_1(\bfB_N)U_N^{\perp}{}^{\star}\cdots P_s(\bfA_N)U_N^{\perp}Q_s(\bfB_N)U_N^{\perp}{}^{\star}\Big) + o(1).
\end{multline*}

\end{corollaire}
\begin{proof} 
Here again, we unfold exactly the same lines of arguments as in the proof of Theorem \ref{th:cfreeness}, obtaining \eqref{eq:lesRS}, which can be written \begin{multline*}
    	\Trn\Big(P_1(\bfA_N)Q_1(U_N\bfB_NU_N^{\star})\cdots P_s(\bfA_N)Q_s(U_N\bfB_NU_N^{\star})\Big) =
    	 \varphi_{\bfA}(P_1)\varphi_{\bfB}(Q_1)\cdots \varphi_{\bfA}(P_s)\varphi_{\bfB}(Q_s)
    \\	+ \Trn\Big(P_1(\bfA_N)U_N^{\perp}Q_1(\bfB_N)U_N^{\perp}{}^{\star}\cdots P_s(\bfA_N)U_N^{\perp}Q_s(\bfB_N)U_N^{\perp}{}^{\star}\Big)
+ \sum_{\emptyset \neq S\subsetneq \{1,\dots,s\}}\Trn(R_S).\end{multline*}
Thanks to \eqref{RStozero}, we already know that $\mathbb{E}[\Trn(R_S)]$ converges to $0$. It remains to prove that $\Trn(R_S)-\mathbb{E}[\Trn(R_S)]$ converges to $0$ almost surely whenever $\emptyset \neq S\subsetneq \{1,\dots,s\}$.

Because $S\neq \{1,\dots,s\}$, each terms $R_S$ contains at least one factor $v_Nv_N^{\star}$, and we can rewrite
$\Trn(R_S)$ as $\Trn(R_S)=\varphi^{v_N}_{U_N,\bfA_N,\bfB_N,v_Nv_N^{\star}}(T)$
for a certain polynomial $T$ in non-commuting variables. Proposition~\ref{Lipschitz} yields the almost sure convergence of $\Trn(R_S)-\mathbb{E}[\Trn(R_S)]$ to $0$.\end{proof}

\subsection{Second-order limit distribution}\label{sec:secondorderlimit}
In this section, we show that the random variable on the right-hand side of the equality in Corollary \ref{cor:fluctuat} is asymptotically Gaussian by using the theory of second-order free probability, presented in section~\ref{sec:fluctutations}.

Under the hypothesis and notations of Theorem~\ref{th:cycliccfreeness}, $U_N$ can be written $v_Nv_N^\star+pV_{N-1}p^\star$ for a certain deterministic $N\times (N-1)$ matrix $p$ and a Haar unitary matrix $V_{N-1}$, as in the proof of Theorem \ref{th:cfreeness}. In particular, we have $U_N^\perp=U_N-v_Nv_N^\star=pV_{N-1}p^\star$. 
Let $k\geq 2$ an integer and $(P_1,Q_1,\ldots,P_s,Q_s)$ be a cyclically alternated sequence of polynomials centered  with respect to $\psi_{\bfA}*\psi_{\bfB}$. We have
\begin{align*}
     &\Trn\Big(P_1(\bfA_N)U_N^{\perp}Q_1(\bfB_N)U_N^{\perp}{}^{\star}\cdots P_s(\bfA_N)U_N^{\perp}Q_s(\bfB_N)U_N^{\perp}{}^{\star}\Big) 
     \\
     =& \Trn\Big(P_1(\bfA_N)pV_{N-1}p^\star Q_1(\bfB_N)pV_{N-1}^\star p^\star\cdots P_s(\bfA_N)pV_{N-1}p^\star Q_s(\bfB_N)pV_{N-1}^\star p^\star\Big)\\
     =&\Trnm\Big(p^\star P_1(\bfA_N)pV_{N-1}p^\star Q_1(\bfB_N)pV_{N-1}^\star p^\star\cdots P_s(\bfA_N)pV_{N-1}p^\star Q_s(\bfB_N)pV_{N-1}^\star \Big).
\end{align*}
Note that under the hypothesis of Theorem~\ref{th:cycliccfreeness}, the matrices $(p^\star P_i(\bfA_N)p)_{1\leq i \leq s}$ and $(p^\star Q_i(\bfB_N)p)_{1\leq i \leq s}$ have each one a second order distribution. Using Theorem~\ref{thm:asymptoticsecondrderfreeness}, we know that $(p^\star P_i(\bfA_N)p)_{1\leq i \leq s}$ and $(V_{N-1}p^\star Q_i(\bfB_N)pV_{N-1}^\star)_{1\leq i \leq s}$ are asymptotically free of second order. In particular, $$\Trn\Big(P_1(\bfA_N)U_N^{\perp}Q_1(\bfB_N)U_N^{\perp}{}^{\star}\cdots P_s(\bfA_N)U_N^{\perp}Q_s(\bfB_N)U_N^{\perp}{}^{\star}\Big) $$ is asymptotically a centered Gaussian variable with explicit variance. Using Corollary~\ref{cor:fluctuat}, we obtain the following.

For any $P^{(1)},\ldots,P^{(r)} \in \mathbb{C}\langle X_k,Y_k:k\in K \rangle$, each one being a product of cyclically alternated elements which are centered with respect to $\psi_{\bfA}*\psi_{\bfB}$, we have the convergence in moments of
\begin{equation*}
    \Big(\mathrm{Tr}_N(P^{(1)}(\bfA_N,U_N\bfB_NU_N^{\star})),\ldots,\mathrm{Tr}_N(P^{(r)}(\bfA_N,U_N\bfB_NU_N^{\star}))\Big)
\end{equation*}
to a Gaussian vector $(Z_1,\ldots,Z_r)$. Setting $P^{(i)} = P^{(i)}_1\cdots P^{(i)}_{\ell(i)}$ the decomposition of $P^{(i)}$ into cyclically alternated and centered factors (with $P^{(i)}_1\in \mathbb{C}\langle X_k:k\in K \rangle$), we have the mean given by $$\mathbb{E}[Z_i]=\varphi_{\bfA}(P_1^{(i)})\varphi_{\bfB}(P_2^{(i)})\cdots \varphi_{\bfB}(P^{(i)}_{\ell(i)})$$ and the covariance matrix $\gamma$ given by
\begin{equation*}
\gamma(i,j) =\delta_{\ell(i),\ell(j)} \sum_{i=1}^{\ell(i)} \psi_{\bfA}* \psi_{\bfB}(P^{(i)}_1P_{1+k}^{(j)})\cdots\psi_{\bfA}*\psi_{\bfB}(P^{(i)}_rP_{r+k}^{(j)}).
\end{equation*}
Finally, by linearity, we deduce the fluctuations of any $\psi_{\bfA_N,U_N\bfB_NU_N^\star}(P)$ in the following proposition (see Section~\ref{sec:fluctutations} for the definition of $(\psi_{\bfA}*\psi_{\bfB})^{(2)}$).

\begin{proposition}\label{prop:gaussian}
Let $(v_N)_{N\geq 2}$ be a sequence of unit vectors of $\mathbb{C}^N$. Let ${\bf A}_N = \{A_N^k: k \in K\}$ and ${\bf B}_N = \{B^k_N:~k \in K\}$ be two sequences of $N\times N$ matrices ($N\geq 2$) as in Theorem~\ref{th:cycliccfreeness}. Let $U_{N} \in U(N)\cap \mathrm{Stab}(v_N)$ be a sequence ($N\geq 2$) of uniform random unitary matrix leaving $v_N$ invariant.

Then, the family of random variables
$$N\Big(\psi_{\bfA_N,U_N\bfB_NU_N^\star}(P)-\psi_{\bf A} * \psi_{\bf B}(P)\Big)_{P\in \mathbb{C}\langle X_k,Y_k:~k\in K\rangle}$$
converges in moments to a Gaussian family of random variables $(Z_P)_{P\in \mathbb{C}\langle X_k,Y_k:~k\in K\rangle}$ with mean
$$\mathbb{E}[Z_P]=\omega_{\bfA} \prescript{\psi_{\bfA}\!\!}{\varphi_{\bfA}\!\!}{}\circledast^{\psi_{\bfB}}_{\varphi_{\bfB}}\omega_{\bfB} (P) $$
and covariance
$$ k_2(Z_P,Z_Q)=(\psi_{\bfA}*\psi_{\bfB})^{(2)}(P,Q).$$
\end{proposition}

\subsection{Asymptotic ordered freeness}\label{sec:matrixmodelordered}


Again, we consider two deterministic sequences ${\bf A}_N = \{A_N^k, k \in K\}$ and ${\bf B}_N = \{B^k_N,~k \in K\}$ of $N\times N$ matrices ($N>1$)  uniformly bounded in operator-norm.
We pick for each integer $N>1$, a pair $(u_N,v_N)$ of orthonormal vectors. Let $U_N\in U(N)\cap \mathrm{Stab}(u_N)$ be a uniform unitary matrix leaving invariant the vectors $u_N$. We set for all $N>1$,
$${\bf A}_N^{u} :=U_N{\bf A}_N U_N^\star.$$
Note that we have
$$
    \varphi^{u_N}_{{\bf A}_N^{u}}=\varphi^{u_N}_{{\bf A}_N},\ \   \text{but}\ \ \varphi^{v_N}_{{\bf A}_N^{u}}\neq\varphi^{v_N}_{{\bf A}_N}.
$$
Indeed, the last linear functional $\varphi^{v_N}_{{\bf A}_N^{u}}$ is random. In fact, the behavior of $\varphi^{v_N}_{{\bf A}_N^{u}}$ is almost surely given by $\psi_{{\bf A}_N}$ as the following proposition shows.

\begin{proposition}\label{prop:macromicro}
For any $P\in \mathbb{C}\langle X_k: k\in K\rangle$, as $N\to \infty$, we have almost surely
$$\varphi^{v_N}_{{\bf A}_N^{u}}(P)=\psi_{{\bf A}_N}(P)+o(1).$$
\end{proposition}
\begin{proof}
In expectation, one has
$\mathbb{E}[\varphi^{v_N}_{{\bf A}_N^{u}}]=\psi_{{\bf A}_N}+O\left(N^{-1}\right)$.
Indeed, the value of $\mathbb{E}[\varphi^{v_N}_{{\bf A}_N^{u}}]$ does not depend of the choice of $v_N$ among the vectors orthogonal to $u_N$ (due to the invariance of $V_N$). In particular, choosing any orthonormal basis $(u_N,v(1),\ldots,v(N-1)),$
we can write
\begin{align*}
   N \psi_{{\bf A}_N}&=\mathbb{E}[N\psi_{{\bf A}_N^{u}}]\\
    &=\mathbb{E}\left[\varphi^{u_N}_{{\bf A}_N^{u}}+\sum_i \varphi^{v(i)}_{{\bf A}_N^{u}}\right]\\
    &=\varphi^{u_N}_{{\bf A}_N}+(N-1) \mathbb{E}[\varphi^{v_N}_{{\bf A}_N^{u}}],
\end{align*}
which implies that $\mathbb{E}[\varphi^{v_N}_{{\bf A}_N^{u}}]=\psi_{{\bf A}_N}+O\left(N^{-1}\right)$.
In order to conclude, it remains to prove that, almost surely,
$$\varphi^{v_N}_{{\bf A}_N^{u}}(P)=\mathbb{E}[\varphi^{v_N}_{{\bf A}_N^{u}}(P)]+o\left(1\right).$$
This almost sure behavior is a consequence of Proposition~\ref{Lipschitz}.\end{proof}


Now, let $V_N\in U(N)\cap \mathrm{Stab}(v_N)$ be a uniform unitary matrix leaving invariant the vector $v_N$ and set
\begin{equation}
\label{eqn:recall}
{\bf B}_N^{v}:=V_N{\bf B}_N V_N^\star.
\end{equation}
Similarly to ${\bf A}_N^{u}$, we have almost surely
$$
\varphi^{u_N}_{{\bf B}_N^{v}}=\psi_{{\bf B}_N}+o(1),\ \ \text{and}\ \ \     \varphi^{v_N}_{{\bf B}_N^{v}}=\varphi^{v_N}_{{\bf B}_N}.
$$
This means that $u_N$ is asymptotically isotropic for ${\bf B}_N^{v}$, while $v_N$ is asymptotically isotropic for ${\bf A}_N^{u}$.
\begin{theorem}
\label{thm:odinde}
With the notations introduced so far, as $N\to \infty$, we have almost surely
\begin{align*}
    &\varphi^{x_N}_{{\bf A}_N^{u},{\bf B}_N^{v}}(P)=\varphi^{x_N}_{{\bf A}_N^{u}}\prescript{}{\varphi^{v_N}_{{\bf A}_N^{u}}\!\!}{*}^{}_{\varphi^{u_N}_{{\bf B}_N^{v}}}\varphi^{x_N}_{{\bf B}_N^{v}}(P)+o(1), \ \ \text{for any}\ \ x_N\in\{u_N,v_N\}.
\end{align*}
In other words, almost surely, the ensembles $\bfA_N^{u,v}$ and $\bfB_N^{u,w}$ are asymptotically $o$-free with respect to $(\varphi^{u_N},\varphi^{v_N})$.
\end{theorem}
\begin{proof}We apply Theorem~\ref{th:cfreeness}; we condition on the value of the random matrix $V_N$, since our estimates are uniform in the ensembles $\bfA_N$ and $\bfB_N$ in a ball for the operator norm of fixed radius (with the notation of Theorem \ref{th:cfreeness}), we get
\begin{align*}
    &\mathbb{E}[\varphi^{u_N}_{{\bf A}_N^{u},{\bf B}_N^{v}}(P)|V_N]=\varphi^{u_N}_{{\bf A}_N^{u}}\prescript{}{\psi_{{\bf A}_N}\!\!}{*}^{}_{\psi_{{\bf B}^v_N}}\varphi^{u_N}_{{\bf B}_N^{v}}(P)+O(N^{-1}).
\end{align*}
It remains to prove that, almost surely,
$$\varphi^{u_N}_{{\bf A}_N^{u},{\bf B}_N^{v}}(P)=\mathbb{E}[\varphi^{u_N}_{{\bf A}_N^{u},{\bf B}_N^{v}}(P)|V_N]+o\left(1\right).$$
It is a consequence of the concentration inequality of
Theorem~\ref{th:concentration_unitary} for $U_N$ since $$U(N)\cap \textrm{Stab}(u_N) \ni U\mapsto \varphi^{u_N}_{U{\bf A}_NU^{\star},{\bf B}_N^{v}}(P)$$is a Lipschitz function (we refer to the proof of Theorem~\ref{th:cfreeness_almost_sure} for the arguments). We get almost surely
$$\varphi^{u_N}_{{\bf A}_N^{u},{\bf B}_N^{v}}(P)=\varphi^{u_N}_{{\bf A}_N^{u}}\prescript{}{\psi_{{\bf A}_N}\!\!}{*}^{}_{\psi_{{\bf B}_N}}\varphi^{u_N}_{{\bf B}_N^{v}}(P)+o(1).$$
Equivalently
$$\varphi^{u_N}_{{\bf A}_N^{u},{\bf B}_N^{v}}(P)=\varphi^{u_N}_{{\bf A}_N^{u}}\prescript{}{\varphi^{v_N}_{{\bf A}_N^{u}}\!\!}{*}^{}_{\varphi^{u_N}_{{\bf B}_N^{v}}}\varphi^{u_N}_{{\bf B}_N^{v}}(P)+o(1)$$
because $\varphi^{v_N}_{{\bf A}_N^{u}}(P)=\psi_{{\bf A}_N}(P)+o(1)$ and $\varphi^{u_N}_{{\bf B}_N^{v}}(P)=\psi_{{\bf B}_N}(P)+o(1)$ thanks to Proposition~\ref{prop:macromicro}.

We reverse the role of $U_N$ and $V_N$ in the previous reasoning in order to get almost surely
$$\varphi^{v_N}_{{\bf A}_N^{u},{\bf B}_N^{v}}(P)=\varphi^{v_N}_{{\bf A}_N^{u}}\prescript{}{\varphi^{v_N}_{{\bf A}_N^{u}}\!\!}{*}^{}_{\varphi^{u_N}_{{\bf B}_N^{v}}}\varphi^{v_N}_{{\bf B}_N^{v}}(P)+o(1).$$
\end{proof}
\begin{remark}
We can leverage this geometric interpretation of the ordered product between pairs of functional to prove associativity of the ordered product. This essentially follows from Theorem \ref{thm:odinde} and the fact that for any random Haar unitary matrix $U_N$ leaving invariant a vector $u_N$ and another random  Haar random unitary matrix $V_N$ leaving invariant a vector $v_N$, independent from $U_N$, one has, almost surely:
\begin{equation*}
    \varphi^{x_N}((\bfA_N^u)^v)=\varphi^{x_N}((\bfA_N^{v})^u)
\end{equation*}
for any sequences of ensembles of deterministic matrices $\bfA_N$ bounded uniformly in $N$.
\end{remark}
\subsection{Asymptotic indented independences}
\label{sec:matrixmodelindented}
In this section, we call a \emph{triad} the data of a triple of unitary vectors $(u,v,w)$ which are  mutually orthogonal.


We consider two deterministic sequences ${\bf A}_N = \{A_N^k, k \in K\}$ and ${\bf B}_N = \{B^k_N,~k \in K\}$ of $N\times N$ matrices ($N\geq 3$)  uniformly bounded in operator-norm.
We pick for each integer $N\geq 3$, a triad $(u_N,v_N, w_N)$. Let $V_N\in U(N)\cap \mathrm{Stab}(u_N)\cap \mathrm{Stab}b(v_N)$ be a uniform unitary matrix leaving invariant the vectors $u_N$ and $v_N$ and $W_N\in U(N)\cap \mathrm{Stab}(u_N)\cap \mathrm{Stab}(w_N)$ be a uniform unitary matrix leaving invariant the vectors $u_N$ and $w_N$. We set for all $N\geq 3$,
$${\bf A}_N^{u,v} :=V_N{\bf A}_N V_N^\star\ \ \text{and}\ \ {\bf B}_N^{u,w}:=W_N{\bf B}_N W_N^\star.$$
Similarly to the previous section (see Proposition~\ref{prop:macromicro}), we have almost surely
$$
    \varphi^{u_N}_{{\bf A}_N^{u,v}}=\varphi^{u_N}_{{\bf A}_N},\ \
\varphi^{v_N}_{{\bf A}_N^{u,v}}=\varphi^{v_N}_{{\bf A}_N},\ \ \text{and}\ \ \    \varphi^{w_N}_{{\bf A}_N^{u,v}}=\psi_{{\bf B}_N}+o(1).
$$
and
$$
    \varphi^{u_N}_{{\bf B}_N^{u,w}}=\varphi^{u_N}_{{\bf B}_N},\ \
\varphi^{v_N}_{{\bf B}_N^{u,w}}=\psi_{{\bf B}_N}+o(1),\ \ \text{and}\ \ \     \varphi^{w_N}_{{\bf B}_N^{u,w}}=\varphi^{w_N}_{{\bf B}_N}.
$$
This means that $v_N$ is asymptotically isotropic for ${\bf B}_N^{u,w}$, while $w_N$ is asymptotically isotropic for ${\bf A}_N^{u,v}$.


\begin{theorem}
With the notations introduced so far, as $N\to \infty$, we have almost surely
\begin{align*}
    &\varphi^{x_N}_{{\bf A}_N^{u,v},{\bf B}_N^{u,w}}(P)=\varphi^{x_N}_{{\bf A}_N^{u,v}}\prescript{}{\varphi^{w_N}_{{\bf A}_N^{u,v}}\!\!}{*}^{}_{\varphi^{v_N}_{{\bf B}_N^{u,w}}}\varphi^{x_N}_{{\bf B}_N^{u,w}}(P)+o(1), \ \ \text{for any}\ \ x_N\in\{u_N,v_N,w_N\}.
\end{align*}
In other words, almost surely, the ensembles $\bfA_N^{u,v}$ and $\bfB_N^{u,w}$ are  asymptotically indented independent.
\end{theorem}


\begin{proof}
Following the proof of Theorem~\ref{th:cfreeness}, we get below the asymptotic $c$-freeness in expectation with respect to $(\varphi^{x_N},\trn)$ with $x_N\in\{u_N,v_N,w_N\}$.
\begin{proposition}
As $N\to \infty$, we have
\begin{align*}
    &\mathbb{E}\left[\varphi^{x_N}_{{\bf A}_N^{u,v},{\bf B}_N^{u,w}}(P)\right]=\mathbb{E}\left[\varphi^{x_N}_{{\bf A}_N^{u,v}}\prescript{}{\psi_{{\bf A}_N}\!\!}{*}^{}_{\psi_{{\bf B}_N}}\varphi^{x_N}_{{\bf B}_N^{u,w}}(P)\right]+O(N^{-1}), \ \ \text{for any}\ \ x_N\in\{u_N,v_N,w_N\}.
\end{align*}
\end{proposition}
\begin{proof}[Sketch of the proof of the proposition]To ease notations, we write $\bfA_N^u$ and $\bfB_N^u$ instead of $\bfA_N^{u,v}$ and $\bfB_N^{v,w}$. We only prove the first item, the two remaining ones are proved in a similar fashion. We prove that for any polynomial $P\in\mathbb{C}\langle X_i,Y_i,~i\in I\rangle$, written as
$$
P=P_1(X)Q_1(Y)\cdots Q_s(X)P_{s}(Y) \quad (\textrm{or } P=P_1(X)\cdots Q_s(Y))
$$
with each of the $P_i$ and $Q_j$ centered with respect to $\psi_{{\bf A}_N}* \psi_{{\bf B}_N}$,
we have
$$
\mathbb{E}\left[\varphi^{u_N}(P({\bfA}^u_N,{\bfB}^u_N))\right] =\mathbb{E}\left[ \varphi^{u_N}(P_1(\bfA^u_N)) \cdots \varphi^{u_N}(Q_p(\bfB^u_N)) \right]+ O(N^{-1})
$$

$$
(\textrm{ or } =\varphi^{u_N}(P_1(\bfA^u_N)) \cdots \varphi^{u_N}(P_p(\bfA^u_N)) + O(N^{-1})).
$$
To that end, we write
$$
V_N = u_Nu_N^{\star} + v_Nv^{\star}_N + p^\star U_{N-2}p,\quad W_N = u_Nu_N^{\star} + w_Nw_N^{\star} + q^\star \tilde{U}_{N-2}q
$$
where $p$ (resp. $q$) is the projector onto $\langle u_N,v_N\rangle^{\perp}$ corestricted to its image (resp. the projector onto $\langle u_N,w_N\rangle^{\perp}$ corestricted to its image) and $U_N$, $\tilde{U}_N$ are Haar unitaries in $U(N-2)$.
As in the proof of Theorem~\ref{th:cfreeness}, we then write
$$
\varphi^{u_N}(P(\bfA^u_N,\bfB^u_N)) = \varphi^{u_N}(P_1(\bfA^u_N)) \cdots \varphi^{u_N}(Q_p(\bfB^u_N)) + R_N
$$
where $R_N$ is a sum of traces of polynomials of the form
$$
u_N^{\star}u_NP_1({\bfA}_N)X_1Q_1({\bfB}_N)X_2\cdots Q_p({\bf B}_N)X_{2p} \quad (\textrm{~or~} u_N^{\star}u_NP_1({\bfA}_N)X_1Q_1({\bfB}_N)X_2\cdots P_p({\bf A}_N))
$$
where $X_i \in\{u_Nu_N^{\star},~~p^{\star}{U}_{N-2}pw_Nw_N^{\star},~~v_Nv_N^{\star}q^{\star}\tilde{U}_{N-2}q,~~w_Nw_N^{\star}p^{\star}{U}_{N-2}^\star p,~~q\tilde{U}_{N-2}^{\star}qv_Nv_N^{\star} \}$ and at least one of the $X$'s is not equal to $u_Nu_N^{\star}$. We immediately see that such polynomial always contains a term $U_{N-2}$, $U_{N-2}^\star$, $\tilde{U}_{N-2}$ or $\tilde{U}_{N-2}^\star$, so we can apply proposition \ref{prop:freenesswithunitaries} as in  the proof of Theorem~\ref{th:cfreeness} and infer that
$$
\mathbb{E}[\varphi^{u_N}(P_1({\bfA}_N)X_1Q_1({\bfB}_N)X_2\cdots Q_p({\bf B}_N)X_{2p}] \quad (\textrm{or~} \mathbb{E}[\varphi^{u_N}(P_1({\bfA}_N)X_1Q_1({\bfB}_N)X_2\cdots P_p({\bf A}_N))]) =O(N^{-1}).
$$\end{proof}
It remains to prove that, almost surely,
$$\varphi^{x_N}_{{\bf A}_N^{u,v},{\bf B}_N^{u,w}}(P)=\mathbb{E}[\varphi^{x_N}_{{\bf A}_N^{u,v},{\bf B}_N^{u,w}}(P)]+o\left(1\right)$$
and
$$\varphi^{x_N}_{{\bf A}_N^{u,v}}\prescript{}{\psi_{{\bf A}_N}\!\!}{*}^{}_{\psi_{{\bf B}_N}}\varphi^{x_N}_{{\bf B}_N^{u,w}}(P)=\mathbb{E}\left[\varphi^{x_N}_{{\bf A}_N^{u,v}}\prescript{}{\psi_{{\bf A}_N}\!\!}{*}^{}_{\psi_{{\bf B}_N}}\varphi^{x_N}_{{\bf B}_N^{u,w}}(P)\right]+o\left(1\right).$$
It is a consequence of the concentration inequality of
Theorem~\ref{th:concentration_unitary} for $V_N$ and $W_N$
(we refer to the proof of Proposition~\ref{Lipschitz} for the arguments). We get almost surely
$$\varphi^{x_N}_{{\bf A}_N^{u,v},{\bf B}_N^{u,w}}(P)=\varphi^{x_N}_{{\bf A}_N^{u,v}}\prescript{}{\psi_{{\bf A}_N}\!\!}{*}^{}_{\psi_{{\bf B}_N}}\varphi^{x_N}_{{\bf B}_N^{u,w}}(P)+o(1).$$
Equivalently
$$\varphi^{x_N}_{{\bf A}_N^{u,v},{\bf B}_N^{u,w}}(P)=\varphi^{x_N}_{{\bf A}_N^{u,v}}\prescript{}{\varphi^{w_N}_{{\bf A}_N^{u,v}}\!\!}{*}^{}_{\varphi^{v_N}_{{\bf B}_N^{u,w}}}\varphi^{x_N}_{{\bf B}_N^{u,w}}(P)+o(1)$$
because $\varphi^{w_N}_{{\bf A}_N^{u,v}}(P)=\psi_{{\bf A}_N}(P)+o(1)$ and $\varphi^{v_N}_{{\bf B}_N^{u,w}}(P)=\psi_{{\bf B}_N}(P)+o(1)$.
\end{proof}

\begin{remark}
We notice an asymmetry in the way the two ensembles of matrices $\bfA_N$ and $\bfB_N$ are defined; the ensemble $\bfA_N$ is rotated by a unitary matrix $V_N$ leaving invariant the plane $\langle u_N,v_N\rangle$, the ensemble $\bfB_N$ is rotated by a unitary matrix $W_N$ leaving invariant the plane $\langle u_N,w_N\rangle$. What happen when we have three ensembles of matrices? Pick a third ensemble of deterministic matrices ${\bf C}_{N}$ bounded in operator-norm uniformly in $N$, and set:
$$
[{\bf C}_N^{w}]^{v} := \{ V_N(W_NC_NW_N^\star) V_N^\star,~C_N \in {\bf C}_N \},\ \  \text{and}\ \ [{\bf C}_N^{v}]^{w} := \{ W_N(V_NC_NV_N^\star) W_N^\star,~C_N \in {\bf C}_N \}.
$$
First of all, albeit the random matrix $V_NW_N$ do only leave invariant $u_N$, it is not distributed uniformly in $U(N)\cap \mathrm{Stab}(u_N)$, since the coefficient of $V_NW_N(w_N)$ along $v_N$ is always zero.

Secondly, the ensemble of matrices $[{\bf C}_N^{w}]^{v}$ and $[{\bf C}_N^{v}]^{w}$has same distributions, almost surely, with respect to the triad $(u_N,v_N,w_N)$, in fact:
\begin{equation*}
    \varphi^{u_N}_{[{\bf C}_N^{w}]^{v}} = \varphi^{u_N}_{[{\bf C}_N^{v}]^{w}}=\varphi^{u_N}_{{\bf C}_N},\ \  \text{and}\ \ \psi_{{\bf C}_N} =\varphi^{x_N}_{[{\bf C}_N^{w}]^{v}} + o(1)= \varphi^{x_N}_{[{\bf C}_N^{v}]^{w}} + o(1), \ \ \text{for}\ \ x_N\in\{v_N,w_N\}.
\end{equation*}
Thirdly, since the convergence we proved in the propositions above are uniform in the ensemble of matrices, one has almost surely
\begin{align*}
\varphi^{x_{n}}_{\bfA_N^{v}, [{\bf C}_N^{v}]^{w}} &= \varphi^{x_N}_{{\bf A}_N^{v}}\prescript{}{\varphi^{w_N}_{{\bf A}_N^{v}}}{*}^{}_{\varphi^{v_N}_{[{\bf C}^{v}_N]^{w}}}\varphi^{x_N}_{[{\bf C}^{v}_N]^{w}}(P) + o(1) \\
&= \varphi^{x_N}_{{\bf A}_N^{v}}\prescript{}{\varphi^{w_N}_{{\bf A}_N^{v}}}{*}^{}_{\varphi^{v_N}_{[{\bf C}^{w}_N]^{v}}}\varphi^{x_N}_{[{\bf C}^{w}_N]^{v}}(P) + o(1).
\end{align*}
For the same reasons, with
$$
{\bf X}_N^{v} := \{\bfA_N^{v }, [{\bf C}_N^{w}]^{v} \},
$$
one has
\begin{equation*}
    \varphi^{x_N}_{{\bf X}^{v}_N, {\bf B}_N^{w}} = \varphi^{x_N}_{{\bf X}_N^{v}}\prescript{}{\varphi^{w_N}_{{\bf X}_N^{v}}}{*}^{}_{\varphi^{v_N}_{{\bf B}^{w}_N}}\varphi^{x_N}_{{\bf B}^{w}_N}(P) + o(1).
\end{equation*}
Hence,
\begin{equation*}
(\varphi_{{\bf A}_N^v,{\bf B}_N^w,[{\bf C}_N^w]^v}^{x_N})_{x=u,v,w} = ((\varphi_{{\bf A}_N}^{x_N})_{x=u,v,w} \leftthreetimes (\varphi_{{\bf B}_N}^{x_N})_{x=u,v,w}) \leftthreetimes (\varphi_{{\bf C}_N}^{x_N})_{x=u,v,w}+o(1).\end{equation*}
\end{remark}

\section{Amalgamation}
\label{sec:amalgamation}
In this section, instead of considering a vector state $\varphi^{v_N}$ along a direction $v_n$ given by a one-dimensional subspace, we want to consider, broadly speaking, the situation where a subspace $V_N$ of $\mathbb{C}^N$ is fixed for each integer $N$, has dimension $d_N \geq 1$ and study in high dimensions the distribution of the projection onto $V_N$ of randomly rotated matrices. This is a natural extension of the questions adressed in the previous sections. The objective is to give matricial models to operator-valued conditional freeness, with amalgamation over an algebra of matrices or over an algebra of idempotents.

The definition of operator-valued conditional freeness is formally very close to its scalar counterpart, except that the state $\psi$ and the state $\varphi$ may take values in a non-commutative algebra $\mathcal{B}$. The appropriate setting is an operator-valued probability space, that is an algebra $\mathcal{A}$ endowed with two commuting actions of an algebra $\mathcal{B}$. In our case those two algebras will be algebras of matrices. The two states $\varphi$ and $\psi$ are linear with respect to these two actions, this means:
\begin{equation*}
    \psi(b_1 a b_2) = b_1 \psi(a) b_2,\quad \varphi(b_1 a b_2) = b_1 \varphi(a) b_2
\end{equation*}
We introduce three models displaying, with the terminology in use in the previous sections, asymptotic conditional operator-valued freeness.

In the first one, for each $N\geq 1$, we fix an orthogonal splitting of the $U(N)$ fundamental representation $\mathbb{C}^{N}$ into summands of fixed dimension $d_N=d$ (independent of $N$ and independent of the summand). In the asymptotic regime $N\to +\infty$, the number of summands in the splitting tends to infinity. Amalgamation in this model is over the algebra of matrices with dimensions $d \times d$. In this model, modulo change of basis, rotating is done by considering a matrix with dimensions $dp_N \times dp_N$, $N = d_Np_N$, as a matrix of size $p_N$
with coefficients matrices of size $d$ and conjugating by a random unitary with dimension $p_N$.

In the second one, the number of summands in the splitting is fixed, say equal to $p$. In the asymptotic regime, the dimensions of each of the summand tends to infinity. Amalgamation is over the algebra of matrices with dimensions $p\times p$. Rotation by a unitary matrix in this models, mean rotating each block of size $d_N$ with a single unitary matrix of the same size, that is conjugating by a random block unitary diagonal matrix.

\subsection{Set-up}

There are various models one can design. In the first one, we start from a given orthogonal splitting of the space $\mathbb{C}^{N}$ as a finite direct sum
\begin{equation}
\label{eqn:decompo}
\mathbb{C}^{N} = \bigoplus_{k=1}^{p_N} V_N^{(k)}
\end{equation}
where for each $k$, $V_N^{(k)}$ is vector space with dimension $d_N$ (with $N=p_Nd_N$). In order to define a conditional expectation, we need the additional data of a an orthonormal basis ${\bf v}_N^{(k)} = (v^{(k)}_1,\ldots,v^{(k)}_n)$ for each subspace $V^{(k)}_N$. We collect these bases in ${\bf v}_N = ({\bf v}_N^{(1)},\ldots,{\bf v}_N^{(k)})$.
For each $N\geq 1$ and given the decomposition \eqref{eqn:decompo}, $\mathcal{M}_{N,N}(\mathbb{C})$ is a $\mathcal{M}_{d_N,d_N}(\mathbb{C})$-bimodule in the following sense:
\begin{equation*}
    m\cdot M \cdot n := \sum_{i,j,k,l} {\bf v}_N^{(i)}\cdot m \cdot {\bf v}_N^{(j)}{}^{\star} \cdot M \cdot {\bf v}_N^{(l)}\cdot n \cdot {\bf v}_N^{(k)}{}^{\star}.
\end{equation*}
Given a set of indeterminates ${\bf X}$, we denote by $\mathcal{M}_{d_N,d_N}(\mathbb{C})\langle {\bf X} \rangle$ all polynomials in the non-commutative indeterminates with coefficients in $\mathcal{M}_{d_N,d_N}(\mathbb{C})$. A monomial in $\mathcal{M}_{d_N,d_N}(\mathbb{C})$ is an alternating word in indeterminates and elements in $\mathcal{M}_{d_N,d_N}(\mathbb{C})$:
$$
M_1X_1M_2\cdots M_k X_k \in \mathcal{M}_{d_N,d_N}(\mathbb{C})\langle {\bf X} \rangle.
$$
As before, given an ensemble of random matrices $\bfA_N$ in $\mathcal{M}_{N}(\mathbb{C})$, we denote by $\varphi^{{\bf v}_N^{(1)}}_{\bfA_N}$ the $\mathcal{M}_{d,d}(\mathbb{C})$-valued bimodule map over $\mathcal{M}_{d,d}(\mathbb{C})\langle {\bf A} \rangle$ computing mixed
of matrices in $\bfA_N$.
Given matrices $A,B \in \mathcal{M}_{p_N,p_N}$ with dimension $p_N\times p_N$, one sets
$$
A \cdot M_N \cdot B =\sum_{i,j,k,l=1}^p A(i,j){\bf v}_N^{(i)}{\bf v}_N^{(j)}{}^{\star}M_N {\bf v}_N^{(k)} {\bf v}_N^{(l)}{}^{\star} B(k,l).
$$
We use the same symbol for two pairs of actions of $\mathcal{M}_{d_N,d_N}(\mathbb{C})$ and $\mathcal{M}_{p_N,p_N}(\mathbb{C})$. There will be no risk of confusions in mathematics expressions because the two are disentangle by the dimensions of the matrices acting; either $(d,d)$ (or $d_N\times d_N$) either $(p,p)$ (or $(p_N,p_N)$). Observe that the two bimodule structures we introduce commute with each other, this will be extensively used in the following paragraphs without recalling it.
\subsection{Square model I}
In the asymptotic regime studied in this section, $d_N = d$ is held constant while $p_N$ tends to infinity.

\begin{figure}[!h]
    \centering
    \includesvg[scale=0.8]{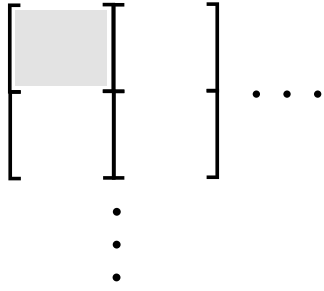}
    \caption{\label{fig:squareun} First model of conditional operator-valued freeness:  $
d_N = d, \quad p_N \to +\infty
$.}
\end{figure}

Given a random matrix $M_N$ with dimensions $(N,N)$, we define the two conditional expectations (with the notations of the previous section):
\begin{align*}
&\varphi^{{\bf v}_N^{(1)}}(M_N):= \frac{1}{d}\mathbb{E}[{\bf v}_N^{(1)}{\bf v}_N^{(1)}{}^{\star}M_N {\bf v}_N^{(1)} {\bf v}_N^{(1)}{}^{\star} ] \in \mathcal{M}_{d,d}(\mathbb{C}),\\
&\psi^{{\bf v} _N}(M_N) := \mathbb{E}[\frac{1}{p_N} \sum_{k=1}^{p_N}{\bf v}_N^{(k)}{\bf v}_N^{(k)}{}^{\star}M_N {\bf v}_N^{(k)}{\bf v}_N^{(k)}{}^{\star}] \in \mathcal{M}_{d,d}(\mathbb{C}).
\end{align*}
The proof of the following proposition relies on Theorem 5.1 in \cite{curran2011asymptotic} and arguments very  similar to the ones developed in the proof of Proposition \ref{th:asymptotic_freeness} (if not verbatim of the last).
\begin{proposition} We pick $\bfA_N$ and $\bfB_N$ be two ensembles of $N\times N$ independent matrices bounded in operator norm. We let $U_N$ be a Haar unitary with dimensions $p_N\times p_N$,  leaving the first vector $\varepsilon_1$ of the canonical basis of $\mathbb{C}^{p_N}$ invariant
	$$
	U_N  = \varepsilon_1\varepsilon_1^{\star} + (I_N- \varepsilon_1\varepsilon_1^{\star})\tilde{V}_N( I_N-\varepsilon_1\varepsilon_1^{\star})
	$$
	where $\tilde{V}_{N}$ is a Haar unitary in $U(p_N-1)$. Then, as $N$ tends to infinity, and almost surely,
	$$
	\varphi^{{\bf v}_N^{(1)}}[P(\bfA_N,~U_N\cdot\bfB_N\cdot U^{\star}_N)] = \varphi_{\bfA_N}^{{\bf v}_N^{(1)}} {}_{\psi_{\bfA_N}^{{\bf v}_N}}\!\!*_{\psi_{\bfB_N}^{{\bf v}_N}} \varphi_{\bfB_N}^{{\bf v}_N^{(1)}} + O(p_N^{-1})
	$$
	for any polynomial $P \in \mathcal{M}_{d,d}(\mathbb{C})\langle X_k,Y_k:~k \in K\rangle$.
\end{proposition}
\begin{proof} We prove the result in expectation. The almost sure statement follows from concentration.
For simplicity, we suppose that ${\bf v}_N^{(k)} = (\varepsilon_{d(k-1)+1},\cdots,\varepsilon_{dk})$, where $(\varepsilon_i)_{1 \leq i \leq N}$ is the canonical basis of $\mathbb{C}^{N}$. We write a matrix of dimensions $p_N\times p_N$ with coefficients in $\mathcal{M}_d(\mathbb{C})$ as an element of the tensor space $\mathcal{M}_{p_N}(\mathbb{C})\otimes\mathcal{M}_{d}(\mathbb{C})$. We pick $P_1,\ldots,P_k$ alternating polynomials in $\mathcal{M}_{d}(\mathbb{C})\langle X_k,Y_k \rangle $, centered with respect to the operator-valued state $\psi^{{\bf v}_N}_{\bfA_N} * \psi^{{\bf v}_N}_{{\bf B}_N}$. So in fact, the polynomials depends on $N$, so we assume further that their coefficients are bounded uniformly in $N\geq 1$ (this typically what happened when taking a an alternated word and centering each $P_i$, since $\psi_{\bfA_N}(P)$ is bounded uniformly in $N$). We have to show that (assuming $P_1 \in \mathcal{M}_{d}(\mathbb{C})\langle X_k\rangle$ for simplicity)
\begin{align*}
    &\mathbb{E}[\varphi^{{\bf v}^{(1)}_N}(P_1(\bfA_N)\cdot U_N\cdot P_2(\bfB_N)\cdot U_N^{\star}\cdots U_N\cdot P_k({\bfB_N})\cdot U_N^{\star})] \\
    \textrm{ or } &\mathbb{E}[\varphi^{{\bf v}^{(1)}_N}(P_1(\bfA_N)\cdot U_N\cdot P_2(\bfB_N)\cdot U_N^{\star}\cdots U_N\cdot P_k({\bfB_N}\cdot U_N^{\star})])
\end{align*}
tends to zero as $N$ goes to infinity. We call $(e^i_j)_{1\leq i,j \leq d}$ the canonical basis of $\mathcal{M}_d(\mathbb{C})$ and write $A(i,N,r)$ for the matrix of the coefficients of $P_i(\bfA_N \textrm{ or } \bfB_N)$, $1 \leq i \leq k$ in the basis element $r=(k,l)$.
Then
\begin{align*}
    &\mathbb{E}[\varphi^{{\bf v}^{(1)}_N}(P_1(\bfA_N)\cdot U_N\cdot P_2(\bfB_N)\cdot U_N^{\star}\cdots U_N\cdot P_k({\bfB_N})\cdot U_N^{\star})] \\
    &\hspace{1cm}=\sum_{r}\mathbb{E}[\varphi^{{\bf v}^{(1)}_N}(A(1,N,r_1)\cdot U_N\cdot A(1,N,r_2)U_N^{\star}\cdots A(k,N,r_k)]e(r_1)\cdots e(r_k))) \\
    &\hspace{1cm}=\sum_{r} \mathbb{E}[\mathrm{Tr}_{p_N}[\varepsilon_1\varepsilon_1^{\star} A(1,N,r_1)U_NA(1,N,r_2)U_N^{\star}\cdots A(k,N,r_k)]]e(r_1)\cdots e(r_k).
\end{align*}
We may thus proceed as in the scalar case to obtain, for each choice of $(r_1,\ldots,r_k)$
\begin{align*}
&\mathbb{E}[\mathrm{Tr}[\varepsilon_1\varepsilon_1^{\star} A(1,N,r_1)U_NA(1,N,r_2)U_N^{\star}\cdots A(k,N,r_k)]] \\
&\hspace{1cm} = \mathrm{Tr}[\varepsilon_1\varepsilon_1^{\star} A(1,N,r_1)]\mathrm{Tr}(\varepsilon_1\varepsilon_1^{\star}(1,N,r_2)]\cdots \mathrm{Tr}[\varepsilon_1\varepsilon_1^{\star}A(k,N,r_k)] + o(1).
\end{align*}
This concludes the proof of asymptotic conditional operator-valued freeness.
\end{proof}
There is not much differences between the scalar case and this first model, due to the fact that amalgamation is over a finite dimensional algebra. However, things get different when considering infinitesimal distribution. In fact, the operator-valued state $\psi^{{\bf v}_N}$ is not tracial anymore, still traciality of $tr_{d_N}$ can be leveraged to give an algebraic recipe for computing next order in the convergence of $\psi^{{\bf v}_N}$.

We suppose existence for $X\in\{\bfA,\bfB\}$ of a $\frac{1}{N}$ expansion of the distribution $\varphi^{{\bf v}_N}_{{\bf X}_N}$,
\begin{equation}
\label{eqn:expansionam}
\psi^{{\bf v}_N}_{{\bf X}_N}(P)=\psi_{{\bf X}}^{{\bf v}}(P) + \frac{1}{p_N} \omega_{{\bf X}}(P) + O(p_N^{-2})
\end{equation}
for any polynomial $P \in \mathcal{M}_d(\mathbb{C})\langle {\bf X}\rangle$.
In the scalar case, we defined the appropriate notion of independence, that we called cyclic-conditional independence, to compute the $\frac{1}{N}$ coefficients in the $\frac{1}{N}$ expansion of the mixed distribution of $\bfA_N$ and $U_N\bfB_N U_N$, with $U_N$ a sequence of random Haar unitary matrices leaving invariant a vector $v_N$. In the operator-valued case, for the model studied in this section, such a $\frac{1}{N}$ expansion will of course holds for the operator-valued mixed distribution of $\bfA_N$ and $U_N\bfB_N U_N$. Nonetheless, the algebraic rules that is the corresponding notion of independence for computing the $\frac{1}{N}$ coefficient are not as \emph{simple} since $\varphi^{{\bf v}_N}$ is not tracial, and thus cyclic words are not enough, to put it roughly.
With the notations of the proof of the previous proposition, consider the alternated word $P_1(\bfA_N)P_2(\bfB_N)P_3(\bfA_N)$. Since $\varphi^{{\bf v}_N}$ is not tracial the limit of $\varphi^{{\bf v}_N}(P_3(\bfA_N)P_1(\bfA_N)P_2(\bfB_N))$ differs from $\varphi^{{\bf v}_N}(P_1(A_N)P_2(B_N)P_3(A_N))$, but still exists (here we use the fact the algebra over which amalgamation is done is finite dimensional).
We may write
\begin{equation*}
    p_N\cdot \mathbb{E}[\varphi^{{\bf v}_N}(P_1(\bfA_N)U_NP_2(\bfB_N)U^{\star}_NP_3(\bfA_N))] = \sum_{r_1,r_2,r_3}\mathbb{E}[\mathrm{Tr}_{p_N}(A_N(1,r_1)U_NB_N(2,r_2)U_N^{\star}A(3,N,r_3))] e_{r_1}e_{r_2}e_{r_3}.
\end{equation*}

Once again, the $\frac{1}{N}$-expansion \eqref{eqn:expansionam} yields a similar $\frac{1}{N}$-expansion for the distribution with respect to the normalized trace $\mathrm{tr}_{d_n}$ of the ensembles $[\bfA]_N=\{A_N(r),A_N\in\bfA_N,1 \leq r \leq d\}$ and $[\bfB]_N=\{B_N(r), B_N\in\bfB_N, 1 \leq r \leq d^2\}$ corresponding to the coordinates of the ensembles $\bfA_N$ and $\bfB_N$ in the basis $\{e_r,1 \leq r \leq d^2\}$:
$$
\mathrm{tr}_{p_N}(P([{\bf X}]_N)) = \psi_{[{\bf X}]}(P) + \frac{1}{p_N}\omega_{[{\bf X}]}(P) + O(p_N^{-2})
$$
where ${\bf X} \in \{\bfA,\bfB\}$.

We write as usual $U_N = \varepsilon_1\varepsilon_1^{\star} + V_N$, substitute this expression in the right hand side of the above expression, use traciality of $\mathrm{Tr}_{p_N}$ to obtain, as $N$ tends to infinity (we drop subscripts to lighten the notations and use bold symbols for the basis elements to improve readability), using the rules of cyclic-conditional independence:
\begin{align*}
    &\sum_{r}\mathbb{E}[\mathrm{Tr}_{p_N}(A(1,N,r_1)U_NB(1,N,r_2)U_N^{\star}A(3,N,r_3))] e_{r_1}e_{r_2}e_{r_3} \\
    &\hspace{1cm} =\{\psi(A(3,r_3)A(1,r_1))\omega(B(2,r_2)) + [\varphi(A(3,r_3)A(1,r_1)) - \psi(A(3,r_3)A(1,r_1))]\varphi(A(2,r_2)) \\
    &\hspace{2cm} - \omega(A(1,r_1)A(3,r_3))\psi(B(2,r_2) \} e_{r_1}e_{r_2}e_{r_3} + o(1) \\
    &\hspace{1cm} ={\bf e_{r_1}}\psi(A(1,r_1)\, \omega(B(2,r_2) {\bf e_{r_2}}\, A(3,r_3))){\bf e_{r_3}} \\
    &\hspace{2cm}+[{\bf e_{r_1}}\varphi(A(3,r_3)\, [\varphi(B(2,r_2)) {\bf e_{r_2}}] \, A(1,r_1)) {\bf e_{r_3}} - {\bf e_{r_1}}\psi(A(1,r_3) \, [\varphi(A(2,r_2)) {\bf e_{r_2}}]\, A(3,r_3)) {\bf e_{r_3}}] \\
    &\hspace{3cm}  -  {\bf e_{r_1}} \omega(A(1,r_1)\psi(B(2,r_2)\, {\bf e_{r_2}} A(3,r_3)){\bf e_{r_3}} \\
    &\hspace{1.5cm}+ o(1)
\end{align*}
where for the last equality, we have used traciality of $\psi$ and $w$ and also $w(1)=0$ (in our model).
From the last equality, we infer that introducing the following bilinear map on $\mathcal{M}_d(\mathbb{C})$,
\begin{equation*}
\varphi_{\bf X_N}(P,Q) = (\mathrm{Tr}_{p_N} \otimes 1)(P({\bf X}_N)\cdot_1 [\varepsilon_1\varepsilon_1^{\star}\otimes I_d] \cdot_1 Q({\bf X}_N)) \in \mathcal{M}_d(\mathbb{C})\otimes \mathcal{M}_{d}(\mathbb{C}),~P,Q \in \mathcal{M}_{d}(\mathbb{C})\langle {\bf X} \rangle
\end{equation*}
and its limit as $N$ goes to infinity, that we denote by $\varphi^{{\bf v}^{(1)}}_{{\bf X}}$, where $\cdot_1$ indicates that we take the product only between the factors in $\mathcal{M}_{p_N}(\mathbb{C})$ of two tensors in $\mathcal{M}_{p_N}(\mathbb{C})\otimes\mathcal{M}_d(\mathbb{C})$ appears as necessary to give a basis-free expression of the limit of $d_N\cdot \psi^{{\bf v}_N}(P_1(A_N)P_2(B_N)P_3(A_N))$:
\begin{align*}
    d_N\psi^{{\bf v}_N}(P_1(\bfA_N)P_2(\bfB_N)P_3( \bfA_N)) = &\psi_{\bfA}(P_1\omega_{\bf B}(P_2)P_3) \\
    +&\varphi_{\bf A}(P_1,P_3) \# \varphi_{\bfB}(P_2)- \psi_{\bfA}(P_1 \varphi_{\bfB}(P_2) P_3) \\
    -&\omega_{\bfA}(P_1 \psi_{\bfB}(P_2) P_3) +o(1)\\
	=&\psi_{\bf A}(P_1(\omega_{\bfB}(P_2)-\varphi_{\bfB}(P_2))P_3) + \varphi_{\bfA}(P_1,P_3)\#\varphi_{\bfB}(P_2)  + o(1)
\end{align*}
with, for $M_1,M_2,N \in \mathcal{M}_d(\mathbb{C})$:
\begin{equation*}
    (M_1 \otimes M_2)\,\# \,N = M_1 \cdot N \cdot M_2 \in \mathcal{M}_d(\mathbb{C}).
\end{equation*}
Note that for any matrix $M \in \mathcal{M}_{d,d}(\mathbb{C})$, $P,Q \in \mathcal{M}_{p_N,p_N}$, one has 
$$
\varphi_{{\bf X}_N}(P,Q\cdot M)= \varphi_{{\bf X}_N}(P,Q)M,~\varphi_{{\bf X}_N}(M\cdot P,Q)= M \varphi_{{\bf X}_N}(P,Q),~\varphi_{{\bf X}_N}(P,M\cdot Q)=\varphi_{{\bf X}_N}(P\cdot M,Q)
$$
$$
\varphi_{{\bf X}_N}(1,P) = \varphi_{{\bf X}_N}(P,1) = \varphi_{{\bf X}_N}(P) .
$$
Of course, $\varphi_{{\bf X}_N}(-,-)$ and $\varphi_{{\bf X}}$ are completely determined by $\varphi_{{\bf X}_N}$ and $\varphi_{\bf X}$. 
The computations we did above generalize to prove the following proposition.
\begin{proposition} With the notations and assumptions introduced so far, for any alternated word $P_1\cdots P_k \in \mathcal{M}_d(\mathbb{C})(\bfA, \bfB)$, with each $P_i$ centered with respect to $\psi=\psi_{\bfA} * \psi_{\bfB}$,
	\begin{equation*}
	\mathbb{E}[d_N \psi_{\bfA_N,U_N\bfB_NU^{\star}_N}^{{\bf v}_N}(P_1\cdots P_k)]
\end{equation*}
converges toward $\omega(P_1\cdots P_k)$ with $\omega:\mathcal{M}_{d}\langle \bfA,\bfB\rangle \to \mathcal{M}_{d,d}(\mathbb{C})$ an $\mathcal{M}_d(\mathbb{C})$-valued bimodule morphism on $\mathcal{M}_d(\mathbb{C})\langle\bfA, \bfB\rangle$ characterized by the following.
If $P_1,\ldots,P_k$ is cyclically alternated then
\begin{equation*}
  \omega(P_1\cdots P_k) = \varphi(P_1)\cdots \varphi(P_k)
\end{equation*}
if it is \emph{not} cyclically alternated ($P_1$,$P_k$ belong to the same subalgebra, $P_2,P_{k-1}$ and so on) then
\begin{equation*}
  \omega(P_1\cdots P_k) = \varphi(P_1,P_k) \# (\varphi(P_2)\cdots \varphi(P_{k-1})) + \psi(P_1(\omega(P_2\cdots P_{k-1})-\varphi(P_2)\cdots \varphi(P_{k-1}))P_k).
\end{equation*}
\end{proposition}
\begin{remark}
As a final word, the above results extend readily to
the case where the random Haar unitary matrices are replaced by quantum Haar unitary matrices and the amalgamation algebra is replace by any Banach algebra $\mathcal{B}$, thanks to \cite{curran2011asymptotic}.
\end{remark}
\subsection{Square model II}
For this model, we choose an orthonormal splitting of $\mathbb{C}^{N}$ into a fixed number of summands, say $p\geq 1$ and let the dimension $d_N$ of each of these subspaces to tend to infinity.

\begin{figure}[!h]
    \centering
    \includesvg{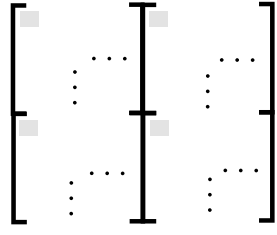}
    \caption{\label{fig:thirdmodel} Second model of conditional operator-valued freeness : $
p_N = p,\quad d_N \to +\infty
$. The grey squares indicated the coefficient computed by the vector state $\psi^{v_N}$. Ellipsis indicates that the asymptotic regime for which each block tends to infinity.}
\end{figure}

We recall the notations of the previous section, for each $1 \leq i \leq p$, ${\bf v}_N^{(i)}$ is a basis of $V_N^{(i)}$.
We define a conditional expectation $\psi^{{\bf v}_N}$ by taking the normalized trace $\mathrm{tr}_{d_N}$ of each:
$$
M_N(j,k)={\bf v}_N^{(j)}{\bf v}^{(j)}_N {}^{\star} \cdot M_N \cdot {\bf v}^{(k)}_N {}^{\star} {\bf v}_{N}^{(k)} \in \mathcal{M}_{d_N,d_N}(\mathbb{C}),~M_N \in \mathcal{M}_{N,N}(\mathbb{C})
$$
with dimension $d_N\times d_N$;
$$
\psi^{{\bf v}_N}(M_N) := (\mathrm{tr}_{d_N}(M_N(i,j))_{1 \leq i,j \leq N} \in \mathcal{M}_{p,p}(\mathbb{C}).
$$
We define a vector-state $\varphi^{{\bf v}_N}$ extracting the upper-left coefficient of each block $M_{N}(i,j)$:
\begin{equation*}
    \varphi^{{\bf v}_N}(M_N) =(\mathrm{Tr}_{d_N}(\varepsilon_1\varepsilon_1{}^{\star}M_N(i,j)))_{1 \leq i,j \leq p},~M_N\in\mathcal{M}_{N}.
\end{equation*}
To ease the exposition of our result, we suppose here forth that $p=2$ and drop references to the basis ${\bf v}_N$ in the notations.
\begin{proposition}
Let $U_{d_N}$ be a random Haar unitary matrix in $U(d_N)$ leaving the first vector $\varepsilon_1$ of the canonical basis of $\mathbb{C}^{d_N}$ invariant:
$$
U_{d_N} = \varepsilon_1\varepsilon_1^{\star} + (I_N- \varepsilon_1\varepsilon_1^{\star}) \tilde{V}_{d_N} ( I_N-\varepsilon_1\varepsilon_1^{\star})
$$
where $\tilde{V}_{d_N}$ is a random unitary matrix with dimensions $(d_N-1) \times (d_N-1)$. 
Let $\bfA_N$ and $\bfB_N$ be two ensembles of deterministic matrices bounded in operator-norm uniformly in $N \geq 1$.
Then as $N$ tends to infinity, 
$
\bfA_N
$
and 
$
U_{d_N}\bfB_NU_{d_N}^{\star}
$
are almost surely asymptotically conditionally free over the algebra $\mathcal{M}_{p,p}(\mathbb{C})$:
\begin{align*}
    &\psi_{\bfA_N,U_N\bfB_NU_N^{\star}} = \psi_{\bfA_N} * \psi_{\bfB_N} + o(1) \\
    &\varphi_{\bfA_N,U_N\bfB_NU_N^{\star}} = \varphi_{\bfA_N} \prescript{}{\psi_{\bfA_N}\!\!}{*}_{\psi_{\bfB_N}} \varphi_{\bfB_N} + o(1)
\end{align*}
\end{proposition}
\begin{proof}
We prove the result holds in expectation, the almost sure statements follow from standard concentration arguments.
We pick an alternating word $(P_1,\ldots,P_k)$ on polynomials in $\mathbb{C}\langle \bfA,\bfB\rangle$, centered with respect to $\psi_{\bfA_N}*\psi_{\bfB_N}$ (Again, they do depend on $N$, and we suppose they have bounded degree and coefficients, which always holds for words obtained by centering with respect to $\psi_{\bfA_N}*\psi_{\bfB_N}$ a fixed, independent of $N$, alternating word $P_1,\ldots,P_k$ on polynomials in $\mathcal{M}(\mathbb{C})\langle \bfA,\bfB\rangle$). We want to prove: 
\begin{align*}
    &\mathbb{E}[\psi_{\bfA_N,U_N\bfB_NU_N^{\star}}](P) = o(1) \\
    &\mathbb{E}[\varphi_{\bfA_N,U_N\bfB_NU_N^{\star}}](P) = \varphi(P_1)\cdots \varphi(P_k) + o(1).
\end{align*}
The first estimate is standard. In fact,
$$
\mathbb{E}[\psi_{\bfA_N,U_N\bfB_NU_N^{\star}}(P)](i,j) = \mathbb{E}\left[\mathrm{tr}_{d_N}\Big((P_1(\bfA_N)U_NP_2(\bfB_N)U_N^{\star}\cdots)(i,j)\Big)\right]
$$
assuming without loss of generality that $P_1\in\mathcal{M}_{d,d}\langle \bfA \rangle$. The above equality follows, because the two bi-modules structure defined over $\mathcal{M}_N(\mathbb{C})$ commute with each other.
We write each polynomial $P_i({\bf X}_N)$ as a linear combination of elements in the  canonical basis of $\mathcal{M}_{p,p}(\mathbb{C})$ with coefficients in $\mathcal{M}_{d_N,d_N}(\mathbb{C})$
$$
P_i({\bf X}_N) = \sum_{a,b}A_i(a,b) \otimes E_{a,b}
$$
where $A_i(a,b)\in\mathcal{M}_{d_N,d_N}(\mathbb{C})$. One has 
\begin{equation*}
\mathbb{E}[\psi_{\bfA_N,U_N\bfB_NU_N^{\star}}(P)](i,j)=\mathbb{E}[\mathrm{tr}_{d_N}(A_1(i,k_1)U_NA_2(k_2,k_3)U_N^{\star}\cdots].
\end{equation*}
Since each $P_i$ is centered, each $A_i(a,b)$ is $1 \leq i\leq k$, $1 \leq a,b \leq d_N$
We may thus conclude by using our statement about scalar asymptotic conditional freeness, 
$$
\mathbb{E}[\mathrm{tr}_{d_N}(A_1(i,k_1)U_NA_2(k_1,k_2)U_N^{\star}\cdots]
$$
tends to $0$, for any choice of $i,k_1,k_2,\ldots$.
We move on to the second statement. By using the notations introduced so far,
\begin{equation*}
\mathbb{E}[\varphi_{\bfA_N,U_N\bfB_NU_N^{\star}}(P)](i,j) = \mathrm{tr}_{d_N}(\varepsilon^{\star}_1\varepsilon_1 A(i,k_1) U_N A_2(k_1,k_2) U_N^{\star} \cdots).
\end{equation*}
We may thus also conclude by using our result about scalar asymptotic conditional freeness.
\end{proof}
\begin{remark}
We have chosen a state vector computing the upper left coefficient of each block $M_N(i,j)$. Asymptotically conditional freeness holds in the more general situation where for each block $M_N(i,j)$ a vector $v_N^{i,j}$ is chosen and the state vector $\varphi$ compute the coefficient of $M_N(i,j)$ along $v_N^{i,j}$.
\end{remark}
\subsection{Rectangular model}
In this section, we elaborate on the previous model dealing with the convergence in high dimensions of the $\mathcal{M}_p(\mathbb{C})$-valued distribution obtained by regular -- square -- decomposition of randomly rotated deterministic matrices and taking the trace of each piece.
In \cite{benaych2009rectangular}, the author initiated studies about convergence in high dimensions of the singular spectrum of random rectangular matrices, that is the convergence of the spectrum of $\sqrt{MM^{\star}}$. The appropriate setting is again operator-valued free probability, with an amalgamation algebra generated by projectors.
We give the definitions introduced in \cite{benaych2009rectangular} which are necessary to state our results.
An operator-valued probability state is 
\begin{enumerate}
\item the data of complex unital algebra $\mathcal{A}$ together with a complete system of orthonormal projectors $(p_i)_{1\leq i \leq n}$;
\begin{equation*}
    \sum_{i=1}^n p_i = 1,\quad p_ip_j = \delta_{i=j}p_i.
\end{equation*}
We set $\mathcal{A}_{kl}:= p_k\mathcal{A}p_l$, $1 \leq k,l \leq n$ and call $\mathcal{D}$ the unital complex algebra generated by the projectors $p_k,~1 \leq k \leq n$.
\item A family of linear functionals $\psi_k:\mathcal{A}_{kk}\to \mathcal{D}$ such that 
\begin{align*}
    &\psi_k(p_k)=1,~\rho_{k} \psi_k(xy) = \rho_l \psi_l(yx), 1 \leq k,l \leq n, x \in \mathcal{A}_{k,l},y \in \mathcal{A}_{l,k}
\end{align*}
where $(\rho_1,\ldots,\rho_n)$ are non-negative real numbers. 
\end{enumerate}
The algebra $\mathcal{A}$ is a $\mathcal{D}$ in a canonical way : the left and right actions of $\mathcal{D}$ on $\mathcal{A}$ are given by left and right translations in $\mathcal{A}$. A conditional expectation $E:\mathcal{A}\to \mathcal{D}$ is built from the sequence of linear functionals $(\psi_k)_{1\leq k \leq n}$:
\begin{equation*}
    \psi(a) = \sum_{k=1}^n\psi_k(a)p_k,\quad a \in \mathcal{A}.
\end{equation*}
Given a sequence of operator-valued probability spaces $(\mathcal{A}_N,\rho^1_N,\ldots,\rho_N^n, \psi_N)$, we say a sequence of ensembles of random variables $(\{a^i_N,i \in I \})_{N\geq 1}$ in $\mathcal{A}_N$ converges in distribution toward an ensemble of random variables $\{a^i,i\in I\}$ in a rectangular probability space $(\mathcal{A}, E, \rho^1,\ldots,\rho^n)$ if, for any polynomial $P\in\mathcal{D}\langle a^i,~i\in I\rangle$, 
$$
\psi_N(P(\{a_N^i,~i\in I\})) \to_{N +\infty} \psi(P(\{a^i,~i\in I\}))
$$
and 
$$
(\rho_N^1,\ldots,\rho_N^n) \to_{N \to +\infty} (\rho^1,\ldots,\rho^n).
$$
\begin{example}
\label{ex:rectmatrix}
Let $\mathcal{M}_N(\mathbb{C})$ be the algebra of square matrices of dimensions $N\times N$. Let $q_1,\ldots,q_N$ a partition of $N$ and set $p_k$ be the projector onto the subspace generated $(\varepsilon_{q_1 + \cdots +  q_{i-1}},\ldots,\varepsilon_{q_1 + \cdots +  q_{i}})$ and set 
$$
\psi_k(M) = \frac{1}{q_k}\mathrm{Tr}(p_kMp_k),\quad M \in \mathcal{M}_N(\mathbb{C}). 
$$
Then with $\rho_k:=\frac{q_k}{N}$, $1 \leq k \leq N$ is a rectangular probability space.
\end{example}
The following proposition is a \emph{consequence of the proof} of Theorem 1.7 in \cite{benaych2009rectangular}: whereas the author does not provide the quantitative estimate on the convergence stated in the following proposition, they appear clearly in the proof of its Theorem, we therefore refer the reader to \cite{benaych2009rectangular} for the proof of the following proposition.
\begin{proposition}[Theorem 1.7 in \cite{benaych2009rectangular}]
\label{prop:rect}
Let $q_1(N)$, $q_2(N)$ be two sequences of integers tending to infinity
Pick two ensembles $\bfA_N$ and $\bfB_N$ of deterministic square matrices with dimensions $q_1(N)\times q_1(N)$ and $q_2(N)\times q_2(N)$, respectively. Pick two ensembles ${\bf D}_N$ and ${\bf E}_N$ of deterministic rectangular matrices with dimensions $q_1(N)\times q_2(N)$ and $q_2(N)\times q_1(N)$, respectively and consider the square matrix of size $N=q_1(N)+q_2(N)$:
\begin{equation*}
    {\bf M}_N := \begin{bmatrix}
                    {\bf A}_N & {\bf D}_N \\
                    {\bf E}_N & {\bf B}_N
              \end{bmatrix}
\end{equation*}
as an element of the rectangular probability space $(\mathcal{M}_N(\mathbb{C}), \psi, p_1,p_2)$ (see Example \ref{ex:rectmatrix}).
Let $V_N(1,1),V_N(2,2)$ be independent Haar unitary matrices with dimensions $q_1(N)\times q_1(N)$ and $q_2(N)\times q_2(N)$ respectively, and set
$$
U_N = \begin{bmatrix}
        V_N(1,1) & 0 \\
        0 & V_N(2,2)
      \end{bmatrix}.
$$
Then, as $N$ tends to infinity, the ensembles ${\bf M}_N$ and  $\{U_N,U_N^{\star}\}$ are asymptotically free with amalgamation over the algebra generated by $\mathcal{D}$.
Besides, for any polynomial $P$ with coefficient in $\langle p_1,p_2\rangle$:
\begin{equation*}
    \mathbb{E}[\psi(P({\bf M}_N,U_N,U_N^{\star}))] = \psi_{{\bf M}_N} * \mathbb{E}[\psi_{U_N}] + O\left(\frac{p_1}{[q_1(N)]^2} + \frac{p_2}{[q_2(N)]^2}\right)
\end{equation*}
where $O\left(\frac{p_1}{[q_1(N)]^2} + \frac{p_2}{[q_2(N)]^2}\right)$ is uniform in the ensembles ${\bf M}_N$ bounded in operator norm by $R$.
\end{proposition}
Let $v_N$ be a sequence of deterministic vectors in $\mathbb{C}^{q_1(N)}$. We define $\varphi^{v_n}: \mathcal{M}_{N}(\mathbb{C}) \to \langle p_1,p_2\rangle$ by
\begin{equation*}
    \varphi^{v_N}(M_N)=\mathrm{Tr}(p_{v_N}M_Np_{v_N})p_1.
\end{equation*}
It is not difficult to check that $\varphi^{v_N}$ is a $\mathcal{D}$-bimodule morphism.
\begin{proposition}
Let 
\begin{equation*}
    {\bf M}_N = \begin{bmatrix} \bfA_N & \bfB_N \\ {\bf C}_N & {\bf D}_N
    \end{bmatrix},\quad {\bf M}^{\prime}_N = \begin{bmatrix} \bfA^{\prime}_N & \bfB^{\prime}_N \\ {\bf C}^{\prime}_N & {\bf D}^{\prime}_N
    \end{bmatrix}
\end{equation*}
be two ensembles of deterministic matrices in the rectangular probability space of example \ref{ex:rectmatrix}.
Let $U_N(1,1)$ be a Haar random unitary matrix stabilizing $v_N$ and $U_N(2,2)$ be an independent Haar random unitary, and set 
$$
U_N = \begin{bmatrix}
        U_N(1,1) & 0 \\
        0        & U_N(2,2)
      \end{bmatrix}.
$$
We assume that as $N\to +\infty$, $q_i(N)$ tends to infinity. We have almost surely, for any polynomial $P\in\mathcal{D}\langle {\bf M},{\bf M}^{\prime}\rangle$
\begin{align*}
    &\psi[P({\bf M}_N,U_N{\bf M}^{\prime}_NU_N^{\star})] = \psi_{{\bf M}_N} * \psi_{{\bf M}^{\prime}_N} + o(1) \\
    &\varphi^{v_N}(P({\bf M}_N,U_N{\bf M}_N^{\prime}U_N^{\star})) = \varphi_{{\bf M}_N}^{v_N} \prescript{}{\psi_{{\bf M}_N}}{*}_{\psi_{{\bf M}^{\prime}_N}} \varphi_{{\bf M}^{\prime}_N}^{v_N} + o(1).
\end{align*}
\end{proposition}
\begin{proof}
We proceed as in the previous section. We pick an alternated word $P_1,\ldots,P_k$ on polynomials in $\mathcal{D}\langle {\bf M},{\bf M}^{\prime}\rangle$ centered with respect to $\psi_{{\bf M}_N} * \psi_{{\bf M^{\prime}}_N}$ (the $P_i's$ do depend on $N$, but has bounded coefficients). We have to show that 
$$
\mathbb{E}[\varphi^{v_N}(P_1\cdots P_k({\bf M}_N,U_N{\bf M}_N U_N^{\star}))] = \varphi^{v_N}(P_1(M_N))\cdots \varphi^{v_N}(P_k({\bf M^{\prime}}_N)) + o(1).
$$
Again, we write $U_N(1,1)=v_Nv_N^{\star} + pV_N(1,1)p^{\star}$ where $V_N(1,1)$ is a $(q_1(N)-1)\times (q_1(N)-1)$ Haar unitary random matrix). 
Since $U_N$ is diagonal and commute therefore with the projectors $p_1,p_2$, 
\begin{align*}
    &\mathbb{E}[\varphi^{v_N}(p_{1}P_1\cdots P_k({\bf M}_N,U_N{\bf M}_N U_N^{\star})p_{1})] \\
    &=\mathbb{E}[\varphi^{v_N}(p_{1}P_1({\bf M}_N)\cdots U_NP_k({\bf M}^{\prime}_N) U_N^{\star}p_{1})] & \hspace{-1cm}(\textrm{ or } \mathbb{E}[\varphi^{v_N}(p_{1}P_1({\bf M}_N)U_N\cdots P_k({\bf M}_N)p_{1})] \\
    &= \mathbb{E}[p_{v_N}P_1({\bf M}_N)p_{v_N}\cdots p_{v_N}P_k({\bf M}^{\prime}_N)p_{v_{N}})] + R(N,{\bf M}_N,{\bf M_N^{\prime}})\\
    &= \varphi^{v_N}(P_1({\bf M}_N))\varphi^{v_N}(P_2({\bf M}^{\prime}_N))\cdots \varphi^{v_N}(P_k({\bf M}^{\prime}_N)) + R(N,{\bf M}_N,{\bf M_N^{\prime}}).
\end{align*}
We have also
\begin{align*}
    &\mathbb{E}[\varphi^{v_N}(p_{i}P_1\cdots P_k({\bf M}_N,U_N{\bf M}_N U_N^{\star})p_{j})] =0
\end{align*}
for any pairs $(i,j)\in  \{1,2\}^2 \backslash \{(1,1)\} $.
We then combine Proposition \ref{prop:rect} and Lemma \ref{lem:almostcentered} following the same line of arguments as in the scalar case to obtain $R(N,{\bf M}, {\bf M^{\prime}})$ tends to zero. This concludes the proof.
\end{proof}
\begin{remark}
The attentive reader may have noticed that in contrary to the scalar case, the vector state $\varphi$ compute the coefficient of a matrix $M_N$ along a single direction $v_N\in \mathbb{C}^{q_1(N)}$. An alternative to $\varphi^{{\bf v_N}}$, close to the setting of the previous section would involve a second linear functional $\varphi^{w_n}$ computing the coefficient of $M_N$ along a direction $w_N \in \mathbb{C}^{q_2(N)}$. With 
$$
\varphi^{v_N,w_N}(M_N):=\varphi^{v_N}(M_N)p_1 + \varphi^{w_N}(M_N)p_2.
$$
it is possible to show, as $N$ tends to infinity, and for any sequences of alternated words on polynomials centered with respect to $\psi_{{\bf M}_N}*\psi_{{\bf M}^{\prime}_N}$ and bounded uniformly in $N$ following the line of arguments of the previous proof, that 
$$
\mathbb{E}[\varphi^{v_N,w_N}(P_1\cdots P_k({\bf M}_N,U_N{\bf M}_N U_N^{\star}))]p_1 = \sum_{x_1,\ldots,x_{k-1} \in \{v,w\}} \mathrm{Tr}(p_{v_N}P_1({\bf M}_N)p_{x^{1}_N}P_2({\bf M^{\prime}_N})p_{x^{2}_N} \cdots P_k({\bf M^{\prime}}_N)p_{v_N}).
$$
\end{remark}

\bibliographystyle{plain}
\bibliography{biblio.bib}

\end{document}